\documentclass{scrartcl}

\addtokomafont{sectioning}{\normalfont\bfseries}
\setkomafont{title}{\normalfont\LARGE}
\setkomafont{subtitle}{\normalfont\Large}
\addtokomafont{pageheadfoot}{\scshape\small}

\usepackage[draft]{fixme}
\usepackage{pdfsync} 
\usepackage[english]{babel} 
\usepackage[utf8]{inputenc}
\usepackage[T1]{fontenc} 
\usepackage[usenames,dvipsnames]{xcolor} 
\usepackage{rotfloat} 
\usepackage{pdfpages} 
\usepackage{graphicx} 
\usepackage[usenames,dvipsnames]{xcolor}
\usepackage[format=plain, indention=0.2cm]{caption}
\usepackage{url} 
\usepackage{booktabs} 
\usepackage{textcomp} 
\usepackage{marvosym} 
\usepackage{multirow} 
\usepackage{makecell} 
\usepackage[]{algorithm2e}
\usepackage{amsfonts, amsmath, amsthm, amssymb, stmaryrd}
\usepackage{mathrsfs}
\usepackage{upgreek}
\usepackage{todonotes}
\usepackage{nicefrac}
\usepackage{txfonts}
\usepackage{csquotes	}

\usepackage[
	backend=bibtex%
	,isbn=false%
	,doi=false%
	,url=false%
	,giveninits=true%
	]{biblatex}
\usepackage{setspace}

\newcommand{\DesiredState}{y_{\operatorname{d}}}

\definecolor{Cn}{HTML}{000000}
\definecolor{Mn}{HTML}{000000}
\definecolor{C}{HTML}{000000}
\definecolor{M1}{HTML}{000000}
\definecolor{M2}{HTML}{000000}

\DeclareMathOperator{\sign}{sgn}
\DeclareMathOperator{\supp}{supp}

\newtheorem{assumption}{Assumption}
\theoremstyle{plain}

\newtheorem{theorem}[assumption]{Theorem}
\theoremstyle{plain}

\newtheorem{remark}[assumption]{Remark}
\theoremstyle{plain}

\newtheorem{definition}[assumption]{Definition}
\theoremstyle{plain}

\newtheorem{lemma}[assumption]{Lemma}
\theoremstyle{plain}

\theoremstyle{plain}

\theoremstyle{plain}

\newcommand{\qand}{\quad\text{and}\quad}

\newcommand{\TimeNode}{e}
\newcommand{\SpaceNode}{e}
\newcommand{\TimeDelta}{\updelta}
\newcommand{\SpaceDelta}{\updelta}
\newcommand{\stiff}{A_h} 
\newcommand{\dualfunction}{K}
\newcommand{\Morrey}{\varPhi}

\newcommand{\adjheat}{L} 
\DeclareMathOperator{\conv}{conv}

\newcommand{\converges}[1]{\stackrel{#1}{\longrightarrow}}

\newcommand{\cC}{{\mathcal{C}}}

\newcommand{\cI}{{\mathcal{I}}}

\newcommand{\cK}{{\mathcal{K}}}
\newcommand{\cL}{{\mathcal{L}}}
\newcommand{\cM}{{\mathcal{M}}}

\newcommand{\cR}{{\mathcal{R}}}

\newcommand{\cU}{{\mathcal{U}}}
\newcommand{\cV}{{\mathcal{V}}}
\newcommand{\cW}{{\mathcal{W}}}

\newcommand{\cY}{{\mathcal{Y}}}

\newcommand{\dd}{{\on{d}}}

\newcommand{\Span}{{\on{span}}}

\newcommand{\pd}{\partial}

\newcommand{\ceq}{\coloneqq}
\newcommand{\qec}{\eqqcolon}
\newcommand{\R}{{\mathbb{R}}}

\newcommand{\N}{\mathbb{N}}

\newcommand{\abs}[1]{\left\lvert#1\right\rvert} 
\newcommand{\nabs}[1]{\lvert{#1}\rvert} 

\newcommand{\norm}[1]{\left\lVert#1\right\rVert}
\newcommand{\nnorm}[1]{\lVert{#1}\rVert}

\newcommand{\ninnerprod}[1]{\langle{#1}\rangle}

\newcommand{\nparen}[1]{(#1)}
\newcommand{\bigparen}[1]{\big(#1\big)}

\newcommand{\Bigparen}[1]{\Big(#1\Big)}

\newcommand{\intervaloo}[1]{\left(#1\right)}

\newcommand{\intervaloc}[1]{\left(#1\right]}

\newcommand{\Set}[1]{\left\{#1\right\}}

\newcommand{\set}[1]{\{#1\}}

\newcommand{\on}[1]{\operatorname{#1}}

\DeclareMathOperator{\diam}{diam}  

\title{Maximal discrete sparsity in parabolic optimal control with measures}
\author{Evelyn Herberg\footnote{Mathematisches Institut, Universität Koblenz-Landau, Campus Koblenz, Universitätsstraße 1, 56072 Koblenz, Germany.}$\,$ , Michael Hinze$^*$, Henrik Schumacher\footnote{Institut für Mathematik, RWTH Aachen University, Templergraben 55, 52062 Aachen, Germany.}}
\date{November 19, 2019}

\begin{document}

\maketitle

\textbf{Abstract.} We consider variational discretization \cite{VD} of a parabolic optimal control problem governed by space-time measure controls.
For the state discretization we use a Petrov-Galerkin method employing piecewise constant states and piecewise linear and continuous test functions in time. For the space discretization we use piecewise linear and continuous functions. As a result the controls are composed of Dirac measures in space-time, centered at points on the discrete space-time grid. 
We prove that the optimal discrete states and controls converge strongly in $L^q$ and weakly-$*$ in $\cM$, respectively, to their smooth counterparts\textcolor{Cn}{, where $q \in (1,\min\{2,1+2/d\}]$ is the spatial dimension}.
Furthermore, we compare our approach to \cite{CasKun}, where the corresponding control problem is discretized employing a discontinuous Galerkin method for the state discretization and where the discrete controls are piecewise constant in time and Dirac measures in space. Numerical experiments highlight the features of our discrete approach.\\



\section{Introduction}  
\label{sec:Intro}
We consider the continuous minimization problem
\begin{equation}
\min_{(u_0,u)\in \mathcal{M}(\bar{\varOmega}_c)\times\mathcal{M}(\bar{Q}_c) }
	J(u_0,u) 
	\coloneqq \tfrac{1}{q} \nnorm{y-\DesiredState}_{L^q(Q)}^{q} 
	+ \alpha \, \nnorm{u}_{\mathcal{M}(\bar{Q}_c)}
	+ \beta \, \nnorm{u_0}_{\mathcal{M}(\bar{\varOmega}_c)},
\tag{$P$}
\label{eq:P}
\end{equation}
where the \emph{state} $y \in L^q(Q)$ solves the following parabolic 
\emph{state equation}
\begin{equation}
\begin{cases}
\partial_t y - \Delta y &=u \qquad\;\;\; \text{in} 
\; Q = \varOmega \times (0,T),\\
y(x,0) &=u_0 \qquad\;\, \text{in} \; \varOmega, \\
y(x,t) &= 0 \qquad\;\;\; \text{on} \; \varSigma = \varGamma \times (0,T),
\end{cases}
\label{eq:PDE}
\end{equation}
with real, regular Borel measures \textcolor{Cn}{$u \in \cM(\bar{Q}_c)$ and $u_0 \in \cM(\bar{\varOmega}_c)$.}
Here $\varOmega \subset \mathbb{R}^d$
is an open, bounded domain with boundary $\varGamma \coloneqq \partial \varOmega$ of  regularity to be discussed later. 
We fix an open, relatively compact interval $I_c \subset \subset I \ceq (0,T)$ 
and a relatively compact subdomain $\varOmega_c \subset\subset \varOmega$
and define the space-time control domain $Q_c \ceq \varOmega_c \times I_c$.
By the Riesz representation theorem (see \cite[Theorem 6.19.]{Rudin}), we may identify the space of regular $\cM(X)$ Borel measures on a subset $X \subset \R^{d+1}$ with the dual space of $\cC_0(X)$, the closure of the space of continuous, compactly supported functions in the supremum norm.
In particular, we have
\begin{align*}
	\mathcal{M}(\bar{\varOmega}_c) = \cC(\bar{\varOmega}_c)^*,
	\quad
	\mathcal{M}(\bar{Q}_c) = \cC(\bar{Q}_c)^*,
	\quad
	\cM(\varOmega) \coloneqq \cC_0 (\varOmega)^*	
	\quad
	\text{and}
	\quad
	\cM(Q) \coloneqq \cC_0 (Q)^*	
	.
\end{align*}
Moreover, the total variation norm for measures coincides with the dual norm:
\begin{equation*}
	\textstyle
	\nnorm{u_0}_{\mathcal{M}(\bar{\varOmega}_c)}
	=
	\sup_{\nnorm{f}_{\cC(\bar{\varOmega}_c)}\leq 1 } \int_{\bar{\varOmega}_c}{f\,\dd u_0}
	\quad 
	\textrm{and} 
	\quad
	\nnorm{u}_{\mathcal{M}(\bar{Q}_c)}
	=
	\sup_{\nnorm{f}_{\cC(\bar{Q}_c)}\leq 1 } \int_{\bar{Q}_c}{f \,\dd u}
	.	
\end{equation*}
Where appropriate, we identify $\cM(\bar{Q}_c)$ with the space $\{ u \in 
\cM(Q) : \supp(u) \subseteq \bar{Q}_c  \}$, and accordingly 
$\cM(\bar{\varOmega}_c)$ with $\{u_0 \in \cM(\varOmega): \supp(u_0) 
\subseteq \bar \varOmega_c\}$. Furthermore, $\alpha>0$, $\beta>0$ are 
given penalty parameters.

The state $y$ is supposed to solve \eqref{eq:PDE} in the following very weak sense, \textcolor{Cn}{equivalent to \cite[Definition 2.1.]{CasKun}} :
\begin{definition}\label{definition1.2}
	A function $y\in L^q(Q)$ is a solution to \eqref{eq:PDE}, if the 
	identity
	\begin{equation}
	\int_{Q}{- \left(\partial_t w + \Delta w \right)y 
		\;dx\,dt} =  \int_{\bar{Q}_c}{w \, du} + \int_{\bar{\varOmega}_c}{w(0) 
		\, du_0} 
	\label{eq:2.1}
	\end{equation}
	holds for all $w \in \cC^\infty( \bar Q)$ with $w(T) = 0$ and $w|_{\varSigma} =0$.
\end{definition}
The solvability of \eqref{eq:2.1} and of problem \eqref{eq:P} have already been established in
\cite[Theorem 2.2.]{CasKun} and \cite[Theorem 2.7.]{CasKun}.
We will also discuss this matter in greater detail later in Section \ref{sec:contoptsys}.
For the moment, we just state that both \eqref{eq:2.1} and \eqref{eq:P} are well-posed with unique solutions provided that (i) $\varOmega$ is a sufficiently regular (e.g. $\varOmega$ is of class $\cC^{1,1}$) and that (ii) \textcolor{Mn}{$q \in (1,\min\{2, 1 +2/d\}]$.}

\bigskip

Optimal control with total variation norm of measures has several applications since its solutions admit a sparsity structure (see, e.g., \cite{Gong, GongH, Stadler}). For example, the problem of identifying the location of a pollution source or the time instance of a pollution based on measurements can be formulated as an optimal control problem with measures. The sparsity structure of the controls here allows for a precise prediction of the pollution source or the time instance at which pollution takes place. 
For the practical implementation, we propose a discrete concept which delivers discrete controls with a maximal sparsity structure, i.e., variational discretization from \cite{VD}, which allows to \textit{control} the discrete structure of the controls through the choice of Petrov-Galerkin ansatz and test spaces in the discretization of the state equation.
The problem can also be discretized by a full discretization approach as is proposed in, e.g., \cite{CasKun}, where piecewise constant controls in time and Dirac measures in space are used. This limits the maximal possible sparsity in this setting to controls which are constant on time intervals. 

For the variational discretization we obtain analogous convergence results as reported in \cite[Theorem 4.3.]{CasKun}. More precisely, we will prove Theorem \ref{thm:DiscVarConvergence}, where \eqref{eq:Pvd} denotes the variational discrete version of the problem \eqref{eq:P}. We denote the implicitly discrete control space $U_h \times \mathcal{U}_{\operatorname{vd}}$, in which \eqref{eq:Pvd} has a unique solution (see Theorem~\ref{thm:exandproj}), and the discretization parameter $\sigma \coloneqq (\tau, h)$, where $\tau$ indicates time and $h$ indicates space (see Section \ref{sec:vardiscr} for more details on the notation).
Our main result reads as follows:

\begin{theorem}\label{thm:DiscVarConvergence}
	Let $\varOmega \subset \R^d$ be a bounded open domain of class $\cC^{1,1}$ and let $q \in (1,2]$ satisfy $q < 1 + 2/d$. 
	For fixed $\sigma$, let $\nparen{ \bar{u}_{0,h},\bar{u}_{\sigma}}$ be the unique solution of problem \eqref{eq:Pvd} that belongs to $U_h \times \mathcal{U}_{\sigma} $,
	and denote the associated state by $\bar{y}_{\sigma}$. 
	
	Then for each sequence of discretizations with $\nabs{\sigma} \to 0$, we have the following convergence properties: 
	\begin{align}
		& \bar{y}_{\sigma} \to \bar y \quad \text{in $L^q(Q)$}, \label{eq:conv1}\\
		& 
		\bar{u}_{\sigma} \stackrel{*}{\rightharpoonup} \bar u
		\quad \text{in $\mathcal{M}(\bar{Q}_c)$}
		\quad
		\text{and}
		\quad
		\bar{u}_{0,h} \stackrel{*}{\rightharpoonup} \bar u_0
		\quad \text{in $\mathcal{M}(\bar{\varOmega}_c)$},
		\label{eq:conv2}
		\\
		&
		\nnorm{\bar{u}_{\sigma}}_{\mathcal{M}(\bar{Q}_c)} 
		\to
		\nnorm{\bar{u}}_{\mathcal{M}(\bar{Q}_c)} 	
				\quad
		\text{and}
		\quad
		\nnorm{\bar{u}_{0,h}}_{\mathcal{M}(\bar{\varOmega}_c)} 
		\to
		\nnorm{\bar{u}_0}_{\mathcal{M}(\bar{\varOmega}_c)},
		 \label{eq:conv3}
	\end{align}	
	where $(\bar{u}_0,\bar{u})$ is the unique solution of \eqref{eq:P} and $\bar{y}$ its associated state.
\end{theorem}
The proof is given on page \pageref{convergenceproof}.

%
Let us briefly comment on related contributions in the literature. 
In \cite{sparseFEM,duality,PV}, control of elliptic partial differential equations with measures is considered and control of parabolic partial differential equations with measures can be found in \cite{CasasClasonKunisch,CasKun,CVZ,KPV}. 
In \cite{CVZ}, the special case of initial data is covered and in \cite{KPV,PV} a priori error estimates are derived. Of particular interest for our approach are the techniques proposed in \cite{duality}, where Fenchel duality is used to set up a predual problem. 
In this context we also consulted the lecture notes by Clason \cite{Clasonlecturenotes}, where Fenchel duality is discussed in detail. 
The variational discrete concept that we use has been introduced in \cite{VD} and the time discrete-scheme in a variational discrete setting has been analyzed in \cite{DHV,Wollner}.

The paper is organized as follows: In Section \ref{sec:contoptsys} we analyze the continuous problem \eqref{eq:P} and its sparsity structure. 
We will then set up the predual problem, show that it has a unique solution and apply the Fenchel duality theorem. The predual problem will be discretized with two different strategies. 
The first one is by variational discretization., We discuss it in Section \ref{sec:vardiscr}, where we derive also a semi-smooth Newton method to solve the variational discrete problem \eqref{eq:Pvd}. The second strategy is a discontinuous Galerkin discretization (see Section \ref{sec:discgalerkin}). The emerging fully discrete problem \eqref{eq:Psigma} is solved analogously to \eqref{eq:Pvd}. Computational results of both approaches are compared in Section \ref{sec:compres}. 
\section{Continuous optimality system}
\label{sec:contoptsys}

In this section, we take a closer look at the solution structure of \eqref{eq:P}. 

\subsection{State Equation}

First, we have to discuss solvability of the state equation \eqref{eq:PDE},
to be interpreted in the form \eqref{eq:2.1}.
To this end, for an arbitrary open domain $\varOmega' \subset \R^d$, we introduce the following anisotropic Sobolev spaces
\begin{align*}
	W^{k,1}_r(\varOmega' \times I)
	\ceq
	\Big\{ w \in L^r(\varOmega' \times I) : \partial_t w,\; \partial_x^1 w, \dotsc, \partial_x^k w \in L^r(\varOmega'\times I)  \Big\},
	\quad
	\text{$k \in \N_0$ and $r \in (1, \infty)$.}
\end{align*}
and define the space
\begin{align*}
	W 
	\ceq
	\Set{ 
		w \in W^{1,1}_2(Q)
		:
		\text{$w |_{\varSigma} =0$, $w(T) = 0$ and $-(\pd_t + \Delta)w \in L^p(Q)$}
	},
\end{align*}
where  $p = (1 - 1/q)^{-1} \in (2,\infty)$ is the Hölder conjugate of $q$.
Because of $L^p(\varOmega) \subset L^2(\varOmega)$, 
the existence and uniqueness theory for weak solutions of parabolic partial differential equations (see e.g.,
\cite[Chapter 7]{MR1625845}) implies that the operator
\begin{align*}
	L \ceq -(\pd_t + \Delta) \colon W \to L^p(Q)
\end{align*}
is an isomorphism of vector spaces. Equipped with the norm
$\nnorm{w}_W \ceq \nnorm{L \,w}_{L^p(Q)}$,
$W$ is a reflexive Banach space and $L$ is an isomorphism of Banach spaces.
The heat operator $\partial_t - \Delta$ 
is the adjoint of $L$ in the sense that $\ninnerprod{L w , y}_{L^p,L^q} = \ninnerprod{L^*y,w}_{W^*,W}$ holds for all $w \in W$ and all $y \in  L^q(Q)$:
\begin{equation*}
\adjheat^* \coloneqq \partial_t - \Delta : L^q(Q) \rightarrow W^*.
\end{equation*}
Since $L$ is continuously invertible, so is its adjoint $L^*$.

\bigskip

Next we show that  $\cM(\bar{\varOmega}_c) \times \cM(\bar{Q}_c)$ embeds continuously into $W^*$.
This will justify putting measures on the right hand side of \eqref{eq:P}.

We choose two further subdomains $\varOmega'$, $\varOmega''  \subset \varOmega$ with smooth boundaries such that $\varOmega_c \subset\subset \varOmega' \subset\subset \varOmega''  \subset\subset  \varOmega$.
By interior regularity estimates (see, e.g., \cite[Theorem~4.3.7]{Krylov}), there exists a $C \geq 0$ such that
\begin{align*}
	\nnorm{w}_{L^p(\varOmega' \times I)}
	+
	\nnorm{\pd_t w}_{L^p(\varOmega' \times I)}	
	+
	\nnorm{\pd_x^2 w}_{L^p(\varOmega' \times I)}		
	\leq C\, \nnorm{L\,w}_{L^p(\varOmega'' \times I)}
	\leq C\, \nnorm{w}_{W}.
\end{align*}
Notice that the norms on the left hand side topologize $W^{2,1}_p(\varOmega' \times I)$.
Thus, the restriction operator $r \colon W \to 
	X_1 \ceq \set{v \in W^{2,1}_p(\varOmega' \times I) 
		: 
		v(T) =0 
	}$ given by $r(w) \ceq w|_{Q_c}$ is continuous.
For $p \in (1 + d/2, \infty)$  (which corresponds to $q \in (1, 1+ 2/d)$), one has a continuous embedding of $W^{2,1}_p(\varOmega' \times I)$ into $\cC(\bar \varOmega' \times \bar I)$ (see \cite[Theorem 10.4]{BesovI}). 
Thus, we have a continuous embedding
$j \colon X_1  \to X_2 \ceq \set{f \in \cC( \bar \varOmega' \times \bar I) : f(T) =0}$.
Utilizing the restriction operator
$s \colon X_2  \to \cC( \bar \varOmega_c) \times \cC(\bar Q_c)$, 
$s(f) \ceq (f(0)|_{\bar \varOmega_c}, f|_{\bar Q_c})$,
we define the continuous linear operator
\begin{align*}
	\Morrey \ceq s \circ j \circ r \colon W \to \cC( \bar \varOmega_c) \times \cC(\bar Q_c).
\end{align*}
This allows us to rewrite \eqref{eq:PDE} and \eqref{eq:2.1} into the following very compact form:
\begin{align}
	\adjheat^*y = \Morrey^* (u_0,u).
	\label{eq:StateEquation}
\end{align}
As we have already established that $\adjheat^*$ is continuously invertible, the well-posedness of \eqref{eq:PDE} and \eqref{eq:2.1} is now evident.

It will soon be essential that $\Morrey^*$ is injective.
Because of
\begin{equation*}
	\ninnerprod{\Morrey^* (u_0,u),w }_{W^*,W} 
	= 
	\textstyle
	\int_{\bar{\varOmega}_c} w(0)|_{ \bar{\varOmega}_c} \, \dd u_0 
	+ 
	\int_{\bar{Q}_c} w|_{ \bar{Q}_c} \, \dd u,
\end{equation*} 
it suffices to show that $\Morrey \colon W \to  \cC(\bar{\varOmega}_c) \times \cC(\bar{Q}_c)$ has dense image.
By virtue of Tietze's extension theorem (see \cite[Theorem~ 35.1.]{Munkres}),
the restriction operator $s$ is surjective. So it suffices to show that $j \circ r$ has dense image.
Next we observe that $\varOmega'$ is an extension domain and that 
each element of $X_2$ can be arbitrarily well approximated by the restriction of a function
$f \in  C^\infty(\R^d \times \bar I) $ with $f(T) =0$.
We pick a mollifier $\varphi \in C^\infty(\R^d)$ with $\supp(\varphi) \subset \varOmega$ and 
$\varphi|_{\bar \varOmega'} = 1$ and put $w(x,t) \ceq \varphi(x) \, f(x,t)$ for all $(x,t) \in Q$.
By construction, we have $w \in W$ and $j \circ r(w) = f$, showing that $\Morrey$ has indeed a dense image. Thus $\Morrey^* \colon \cM(\bar{\varOmega}_c) \times \cM(\bar{Q}_c) \to W^*$ is injective.

\subsection{Fenchel Duality}

As we want to solve our problem numerically by utilizing Newton-based methods, we have to cope with the fact that, because of $q<2$, the contribution $y \mapsto \frac{1}{q} \, \nnorm{y-\DesiredState}_{L^q}^q$ in the objective function $J$ is not twice differentiable.
Even worse, $u \mapsto \nnorm{u}_{\cM(\bar Q_c)}$ and $u_0 \mapsto \nnorm{u_0}_{\cM(\bar \varOmega_c)}$ are not differentiable at all. 
Fortunately, as demonstrated in \cite[Chapter 2.1.]{duality}, Fenchel duality can help here:
It allows us to transform problem \eqref{eq:P} into an optimization problem \eqref{eq:Pdual} that enjoys sufficient differentiability to make it amenable to the semi-smooth Newton method.
As a convenient side effect, this will also allow us to show that \eqref{eq:P} has a unique solution.

\bigskip

We define the problem \eqref{eq:Pdual} as follows; as it will turn out in Theorem~\ref{thm:fenchel}, this problem is indeed the Fenchel \emph{pre}dual of \eqref{eq:P}: 
\begin{equation}
\min_{w \in W}{ \; \dualfunction (w) } \coloneqq F(w) + G(\varPhi \, w),
\tag{$P^*$}
\label{eq:Pdual}
\end{equation}
where $F \colon W \to \R$ and $G \colon \cC(\bar{\varOmega}_c) \times \cC(\bar{Q}_c) \rightarrow \bar{\R} \coloneqq \mathbb{R} \cup \{\infty\}$
are given by
\begin{align*}
	F(w) 
	\ceq 
	\tfrac{1}{p} \nnorm{\adjheat w}_{L^p(Q)}^p + \ninnerprod{ \adjheat w,\DesiredState}_{L^p(Q),L^q(Q)}	
	\quad
	\text{and}
	\quad 
	G(f_0 , f)
	\ceq
	\begin{cases}
		0, 
		& \text{\textcolor{Cn}{if} $\nnorm{f_0}_{\cC(\bar \varOmega_c)} \leq \beta$ and $\nnorm{f}_{\cC(\bar Q_c)} \leq \alpha$},
		\\
		\infty,
		&\text{else.}
	\end{cases}
\end{align*}
%
%
%
%
%
\begin{theorem}\label{thm:PdualExistenceUniqueness}
	Let $q \in (1,2]$ satisfy $q < 1 + 2/d$. Then problem \eqref{eq:Pdual} has a unique solution $\bar{w} \in W$. 
\end{theorem}
\begin{proof}
	Let $\{w_k\}_k \subset W$ be a minimizing sequence so that 
	\begin{equation*}
	\lim_{k \rightarrow \infty }{K(w_k)} = \inf_{w \in W} K(w) \qec \underline{K}.
	\end{equation*}
	From $K(0) =0 $, we know that $\underline{K} \leq 0$, which allows us to assume without loss of generality that $K(w_k)\leq 0 < \infty$ for all $k$ and hence $G(\Morrey \, w_k)= 0$ for all $k$. With Hölder's inequality and Young's inequality in the form $-a \, b \geq - \tfrac{1}{p}\, a^p - \tfrac{1}{q}\, b^q$, we see 
	\begin{align*}
	K(w_k) = F(w_k) 
	&\geq  \tfrac{1}{p} \nnorm{\adjheat w_k}_{L^p(Q)}^p - \nnorm{\adjheat w_k}_{L^p(Q)} \nnorm{\DesiredState}_{L^q(Q)}
	\geq - \tfrac{1}{q}\nnorm{\DesiredState}_{L^q(Q)}^q
	\end{align*}
	and hence $\underline{K}> -\infty$.
	For all $k$ we have $F(w_k) \leq K(w_k)\leq 0$, which implies that  $\{w_k\}_k$ is a bounded sequence in $W$. 
	Recall that $W$ is reflexive. Hence the bounded sequence $\{w_k\}_k$ admits a weakly convergent subsequence: $w_{k'} \rightharpoonup \bar{w}$ in $W$ as $k' \to \infty$. 
	Likewise, we have $\Morrey \, w_{k'} \rightharpoonup \Morrey \, \bar{w}$ in $\cC(\bar{\varOmega}_c) \times \cC(\bar{Q}_c)$.
	As the indicator function of a closed, convex set, $G$ is convex and weakly lower semi continuous
	Thus, we are lead to $G(\Morrey \, \bar{w})=0$.	
	Also $F$ is weakly lower semi continuous, so we obtain:
	\begin{equation*}
	\underline{K} \leq K(\bar{w}) = F(\bar{w}) \leq \liminf_{k' \rightarrow \infty } F(w_{k'}) = \lim_{k' \rightarrow \infty} K(w_{k'}) = \lim_{k \rightarrow \infty} K(w_k) = \underline{K}.
	\end{equation*}
	This shows that $\bar w$ is a minimizer of \eqref{eq:Pdual}. 
	Moreover, 
	$G$ is convex
	and
	$F$ is strictly convex ($L$ is injective and $2 \leq p < \infty$),
	hence $K = F + G \circ \Morrey$ is also strictly convex. Thus, there cannot be more than one solution of \eqref{eq:Pdual}.
\end{proof}

%
%
%
%

Now we recall the Fenchel duality theorem in a similar notation as in \cite{Clasonlecturenotes} and \cite[Chapter 1.1.3.]{duality}. We also refer to \cite[Chapter III.4.]{Ekeland}.
For a convex, lower semi continuous functional $F \colon R \rightarrow \bar{\mathbb{R}}$ 
on a normed space $R$ with $\inf_{r \in R} F(r) < \infty$,
we define its \emph{Fenchel conjugate} by
\begin{equation*}
	F^* \colon R^* \rightarrow \bar{\mathbb{R}}, 
	\quad 
	F^*(\varrho) \ceq \sup_{r \in R}{ \ninnerprod{\varrho,r}_{R^*,R}}-F(r).
\label{eq:fenchelconj}
\end{equation*} 

\begin{theorem}[Fenchel Duality]\label{thm:FenchelDualityTheorem}
Let $R$ and $S$ be normed spaces with topological duals $R^*$ and~$S^*$ and let $\varLambda : R \rightarrow S$ be a continuous linear operator. 
Let $F \colon R \rightarrow \bar{\mathbb{R}}$ and $G \colon S \rightarrow \bar{\mathbb{R}} $ be convex lower semi continuous functionals and suppose that $F$ and $G$ are not identically equal to~$\infty$.
Consider the \emph{primal problem}
\begin{align}
	\inf_{r \in R}{F(r) + G(\varLambda \, r)}
	\label{eq:FenchelPrimalProblem}
\end{align}
and the \emph{dual problem}
\begin{align}
	\sup_{\sigma \in S^*}{ -F^*(\varLambda^* \sigma)-G^*(-\sigma)}.
	\label{eq:FenchelDualProblem}
\end{align}
Suppose the following two conditions are fulfilled:
\begin{itemize}
	\item The primal problem \eqref{eq:FenchelPrimalProblem} has at least one solution.
	\item The \emph{regular point conditions} is fulfilled, i.e., there exists an $r_0 \in R$, such that 
	$F(r_0)<\infty $
	and
	$G(\varLambda \,  r)<\infty$ for all $r$ in a sufficiently small neighborhood of $r_0$.
\end{itemize}
Then also the \emph{dual problem}
has at least one solution and one has the identity
\begin{equation}
	\min_{r \in R}{F(r)+G(\varLambda\, r)} 
	= 
	\max_{\sigma \in S^*}{-F^*(\varLambda^* \sigma)-G^*(-\sigma)} .
	\label{eq:fenchel}
\end{equation}
Furthermore, for $r \in R$ and $\sigma \in S^*$, the following \textcolor{Mn}{three} statements are equivalent:
\begin{enumerate}
	\item $r$ is a solution of the primal problem \eqref{eq:FenchelPrimalProblem}
	and $\sigma$ is a solution of the dual problem \eqref{eq:FenchelDualProblem}.
	\item $F(r)+G(\varLambda \, r) = -F^*(\varLambda^* \sigma)-G^*(- \sigma)$.
	\item $\varLambda^* \sigma \in \partial F(r) $ and $ - \sigma \in \partial G(\varLambda \,  r). $ \label{eq:Fencheldifferentials} 
\end{enumerate}
\end{theorem}

We are going to apply the Fenchel duality theorem to $R =W$, $S \ceq \cC(\bar{\varOmega}_c) \times \cC(\bar{Q}_c)$, and \textcolor{Mn}{$\varLambda = \Morrey$}.
To this end, we show first that \eqref{eq:Pdual} and \eqref{eq:P} are dual to each other.
%
%
%
%
%
\begin{theorem}\label{thm:FenchelDuality}
	The Fenchel dual problem of \eqref{eq:Pdual} coincides with  \eqref{eq:P}. 
	\label{thm:fenchel}
\end{theorem}
\begin{proof}
	It suffices to show that $J(u_0,u) = F^*(\textcolor{Mn}{\Morrey^*} (u_0,u)) + G^*(-u_0,-u)$.
	
	In order to calculate $F^*$, we use the equivalence of the following statements on $w \in W$ and $\xi \in W^*$:
	\begin{equation}
	F^*(\xi)
	= 
	\ninnerprod{\xi,w}_{W^*,W}  - F(w)
	\quad 
	\textrm{if and only if} 
	\quad 
	\xi \in \partial F(w).
	\label{eq:fenchelsubdiff}
	\end{equation}
	Here $\partial F(w)$ denotes the subdifferential of the convex functional $F$.
	Since $F$ is Fréchet differentiable, we have $\partial F(w) = \set{ DF(w)}$.
	Hence $\xi \in \partial F(w)$ is given by
	$\xi = \adjheat^* \nparen{ \nabs{\adjheat w}^{p-2} \adjheat w + \DesiredState }$.
	Solving for $w$ leads to
	\begin{equation}
	w
	= 
	\adjheat^{-1} \bigparen{\sign \left(\adjheat^{-*} \xi - \DesiredState\right) \, \nabs{\adjheat^{-*} \xi - \DesiredState}^{1/(p-1)}}.
	\label{eq:upartialF}
	\end{equation}
	Substituting this into 
	\eqref{eq:fenchelsubdiff} 
	and utilizing $\ninnerprod{\xi ,\adjheat^{-1}z}_{W^*,W} = \ninnerprod{z,\adjheat^{-*} \xi}_{L^{p}(Q),L^q(Q)} $ for all $z \in L^p(Q)$, we derive 
	\begin{align*}
	F^*(\xi)
	&= 
	\ninnerprod{ \xi,w }_{W^*,W} -  \tfrac{1}{p} \nnorm{ \adjheat w }_{L^p(Q)}^p - \ninnerprod{ \adjheat w,\DesiredState }_{L^p(Q),L^q(Q)}    
	\\
	&= 
	\ninnerprod{\sign \bigparen{\adjheat^{-*} \xi - \DesiredState} \, |\adjheat^{-*} \xi - \DesiredState|^{1/(p-1)},\adjheat^{-*} \xi - \DesiredState}_{L^p(Q),L^q(Q)} -  \tfrac{1}{p} \nnorm{ \adjheat^{-*} \xi - \DesiredState }_{L^q(Q)}^q 
	\\
	&=\tfrac{1}{q} \| \adjheat^{-*} \xi - \DesiredState\|_{L^q(Q)}^q.
	\end{align*}
	
	In order to derive $G^*(u_0,u)$ we can interpret $G(f_0,f)$, as consisting of two summands, which represent an indicator function with only one constraint respectively, where we want to use the following notation:
	\begin{equation*}
	\ell_{\alpha}(0,f) + \ell_{\beta}(f_0,0) \coloneqq G(f_0,f).
	\end{equation*}
	Here, we make use of $ (f_0,f)\in \cC(\bar{\varOmega}_c)  \times \cC(\bar{Q}_c) $ and $ \left(\cC(\bar{\varOmega}_c)  \times \cC(\bar{Q}_c) \right)^*  =\cM(\bar{\varOmega}_c)  \times \cM(\bar{Q}_c) $. From \cite[Theorem 2.2.8]{Schiro} we know that for $\tilde{u} =   ( u_0 , 0 ) + ( 0 , u ) \in \cM(\bar{\varOmega}_c)  \times \cM(\bar{Q}_c)$ it holds that :
	\begin{equation*}
	G^*(\tilde{u}) = \bigparen{ \ell_{\alpha}+\ell_{\beta} } ^* (u_0, u)   = \ell_{\alpha}^* ( 0, u)  + \ell_{\beta}^* ( u_0, 0).
	\end{equation*} 

	Looking at both conjugates separately, we can use \eqref{eq:fenchelsubdiff} and derive:
	\begin{align*}
	\ell_{\alpha}^* ( 0, u) 
	&= \sup_{(0,f)\in \cC(\bar{\varOmega}_c)  \times \cC(\bar{Q}_c)} 
	\ninnerprod{ (0,u),(0,f)} 	- \ell_{\alpha}(0,f)\\
	&=\sup_{f \in  \cC(\bar{Q}_c),\, \nnorm{f}_{\cC(\bar{Q}_c)}\leq \alpha}\, 
	\textstyle 
	\int_{\bar{Q}_c}  f \, \dd u 
	\; = \; \alpha \, \nnorm{u}_{\cM(\bar{Q}_c)},
	\end{align*}
	where 
	$\ninnerprod{(u_0,u),(f_0,f)} \coloneqq \int_{\bar{\varOmega}_c} f_0 \, \dd u_0 + \int_{\bar{Q}_c} f \,\dd u$. 
	Analogously, we observe that $\ell_{\beta}^*( u_0, 0)  = \beta \,  \nnorm{u_0}_{\mathcal{M}(\bar{\varOmega}_c)}$. 
	Assembling this information, we obtain 
	\begin{align*}
		G^*(\tilde{u}) = \alpha \, \nnorm{u}_{\cM(\bar{Q}_c)} + \beta \,  \nnorm{u_0}_{\cM(\bar{\varOmega}_c)}.
	\end{align*}
\end{proof}

\begin{theorem}\label{thm:PisUniquelySolvable}
Problem \eqref{eq:P} has a unique solution $(\bar u_0,\bar u)$, which is characterized by
\begin{align}
	\Morrey^* (\bar u_0, \bar u) = DF(\bar w) = \adjheat^* \nparen{ \nabs{\adjheat \bar w}^{p-2} \adjheat \bar w + \DesiredState} \quad \textcolor{Mn}{ \text{and} \quad -(\bar{u}_0,\bar{u}) \in \partial G (\Morrey (\bar{w})),}
	\label{eq:RecoverControl}
\end{align}
where $\bar w$ is the unique solution of \eqref{eq:Pdual}.
The optimal state $\bar y$ can be retrieved from $\bar w$ via
\begin{align*}
	\bar y = \adjheat^{-*} \Morrey^* (\bar u_0, \bar u) = \nabs{\adjheat \bar w}^{p-2} \adjheat \bar w + \DesiredState.
\end{align*}
\end{theorem}
\begin{proof}
	We have seen already in the proof of \textcolor{Mn}{Theorem} \ref{thm:PdualExistenceUniqueness} that 
	both	
	$F$ and $G$ are convex and lower semi continuous and that $F$ is even strictly convex. Furthermore we set $\varLambda = \Morrey$. 
	By \textcolor{Mn}{Theorem} \ref{thm:PdualExistenceUniqueness}, the problem \eqref{eq:Pdual} has a at least one minimizer.
	For $w_0 = 0$, we have $\Morrey \, w_0 = 0$ and $\adjheat w_0 = 0$, showing that $F(w_0) < \infty $ and $G(\Morrey\, w_0)< \infty$. 
	Due to $\alpha, \beta >0$, $G$ is continuous at $\Morrey \, w_0$ so that also the regular point condition is fulfilled.
	Thus, we may apply the Fenchel duality theorem, which implies that \eqref{eq:P} has at least one solution. 
	Since \textcolor{Mn}{$\adjheat^{-*} \Morrey^*$ is injective and $q>1$}, $J$ is strictly convex, hence there is only one solution $(\bar u_0,\bar u)$.
	By the the third condition from \textcolor{Mn}{Theorem} \ref{thm:FenchelDualityTheorem},
	each solution has to satisfy
	$\Morrey^* (\bar u_0,\bar u) \in \partial F(\bar w) =  \set{DF(\bar w)}$ \textcolor{Mn}{and $-(\bar{u}_0,\bar{u}) \in \partial G (\Morrey (\bar{w}))$},
	where $\bar w$ is a solution of \eqref{eq:Pdual}.
\end{proof}

\subsection{Sparsity Structure}

We recall the optimality conditions for the solution $\bar y$ of \eqref{eq:P} from \cite[Theorem 3.1.]{CasKun} and the resulting sparsity structure \cite[Corollary 3.2.]{CasKun} of the optimal controls $\nparen{\bar{u}_0,\bar{u}}$. 
These results can be directly transferred to our setting since the latter is a special case of the formulation in \cite{CasKun}.

\begin{lemma} \label{lem:optimalitycond}
Let $\nparen{\bar{u}_0,\bar{u}}$ denote a solution to $(P)$ with associated state $\bar{y}$.
Denote by $\bar w \in W$ the unique solution of
\begin{align*}
	L \, \bar w = \nabs{\bar{y}-\DesiredState}^{q-2} (\bar{y}-\DesiredState)
	\in L^p(Q),
\end{align*}
\textcolor{Mn}{which follows from the first part of \eqref{eq:RecoverControl} by solving for $L \, \bar w$.}
\textcolor{Mn}{From the second part of \eqref{eq:RecoverControl} we have that $-(\bar{u}_0,\bar u) \in \partial G(\Morrey(\bar w)) \Leftrightarrow \Morrey (\bar w) \in \partial G^*(-(\bar u_0 , \bar u))$, hence} $\bar w$ satisfies
\begin{align*}
	\textstyle
\int_{\bar{\varOmega}_c}{\bar{w}(0)\, d\bar{u}_0} + \beta \, \nnorm{\bar{u}_0}_{\mathcal{M}(\bar{\varOmega}_c)}=0,
	\quad	
	\int_{\bar{Q}_c}{\bar{w} \, d\bar{u}} + \alpha \, \nnorm{\bar{u}}_{\mathcal{M}(\bar{Q}_c)}=0.
\end{align*}
and
\begin{align*}
\setlength{\arraycolsep}{0.2em}
\begin{array}{rclcrcl}
	\nabs{\bar{w}(x,t)} &= 		&\alpha 	& \quad\text{for all}\quad &(x,t) 	&\in &\bar{Q}_c \cap \supp(\bar u),\\
	\nabs{\bar{w}(x,t)} &\leq 	&\alpha 	& \quad\text{for all}\quad &(x,t) 	&\in &\bar{Q}_c \setminus \supp(\bar u),\\
	\nabs{\bar{w}(x,0)} &= 		&\beta 	& \quad\text{for all}\quad &x 		&\in &\bar{\varOmega}_c \cap \supp(\bar u_0),\\
	\nabs{\bar{w}(x,0)} &\leq 	&\beta 	& \quad\text{for all}\quad &x 		&\in &\bar{\varOmega}_c \setminus \supp(\bar u_0).
\end{array}
\end{align*}
\end{lemma}
\begin{remark} 
	Under the assumptions of the previous Theorem we have the following sparsity structure:
	\begin{align*}
		\supp(\bar{u}^+) &\subset \set{(x,t) \in \bar{Q}_c : \bar{w}(x,t)=-\alpha },
		&
		\supp(\bar{u}^-) &\subset \set{(x,t) \in \bar{Q}_c : \bar{w}(x,t)=+\alpha },\notag\\
		\supp(\bar{u}_0^+) &\subset \set{x \in \bar{\varOmega}_c : \bar{w}(x,0)=-\beta }, 
		&
		\supp(\bar{u}_0^-) &\subset \set{x \in \bar{\varOmega}_c : \bar{w}(x,0)=+\beta }, \notag	
	\end{align*}
	where $\bar{u}= \bar{u}^+ - \bar{u}^-$ and $\bar{u}_0= \bar{u}_0^+ - \bar{u}_0^-$ are the Jordan decompositions. Let us note that $\bar{w}$ is the adjoint in \eqref{eq:P}.
	\label{sparsity}
\end{remark}
If one considers it as the generic case that the function $\bar w$ is not constant on sets of measure greater than zero, 
the controls have support sets of measure zero. This is our motivation to propose a discretization strategy which reflects this behavior on the discrete level in space and time.

\section{Variational Discretization}
\label{sec:vardiscr}
Here we want to achieve the desired maximal discrete sparsity, i.e., Dirac-measures in space-time, by choosing the Petrov-Galerkin-ansatz and -test space that will induce this structure. The variational discretization concept was introduced in \cite{VD} and its key feature is to not discretize the control space. Instead, via the discretization of the test space and the optimality conditions, an implicit discretization of the control is achieved. This is how we \textit{control} the discrete structure of the controls. Looking at the relation \eqref{eq:upartialF} between the optimal adjoint state $\bar{w}$ and $\bar{\xi}=\Morrey^* (u_0,u)$, it is obvious, that the discrete structure of the test space affects the structure of the optimal controls  $(u_0,u) \in \mathcal{M}(\bar{\varOmega}_c) \times \mathcal{M}(\bar{Q}_c)$. This motivates the following choice of discrete spaces: We define the state space $\cY_{\sigma}$ consisting of continuous and piecewise linear functions in space and piecewise constant functions in time, whereas we define the test space $\cW_{\sigma}$ consisting of continuous and piecewise linear functions in space and time.

In the following we discretize \eqref{eq:P} and analyze the structure of the controls. \textcolor{Cn}{We will see that the above choice of discrete state and test spaces in combination with the optimality system of the discrete problem induces Dirac measures in space and time for the controls.} Afterwards, we discuss the existence and uniqueness of solutions to the variational discrete problem. We then prove the convergence properties stated in Theorem \ref{thm:DiscVarConvergence}. Afterwards we discretize \eqref{eq:Pdual}, reformulate the problem equivalently, and derive an optimality system by a Lagrange approach. 
If the necessary conditions are fulfilled, we can apply a semismooth Newton method to solve the optimality system. Utilizing Fenchel duality, we can finally calculate the optimal solution of \eqref{eq:P}.


As a first step to characterizing the discrete spaces, we have to set up the space-time grid. 
Define the partition $0 =t_0 < t_1 <\ldots <t_{N_{\tau}} =T$. 
The time interval $I$ is \textcolor{C}{decomposed in subintervals $I_k \coloneqq \intervaloc{t_{k-1},t_k}$ for $k=1,\ldots,N_{\tau-1}$ and $I_{N_{ \tau}}\coloneqq \intervaloo{t_{N_{\tau}-1},t_{N_{\tau}}}$. 
We point out that we defined the intervals $I_k$ such that they cover the full time interval $I$, which is crucial because we deal with measures that can be supported on isolated points.} 
The temporal grid size is denoted by $\tau = \max_{\textcolor{Mn}{1\leq k \leq N_{\tau}}} {\tau_k} $, where $\tau_k \coloneqq t_{k}-t_{k-1}$. 
Let $\mathcal{K}_h$ be a finite triangulation of $\varOmega$ with grid size $h= \max_{K \in \mathcal{K}_h} \diam(K) $. 
We set $\bar{\varOmega}_h \ceq \bigcup_{K \in \mathcal{K}_h}K$ and denote by $\varOmega_h$ the interior,
by $\varGamma_h$ the boundary of $\bar{\varOmega}_h$, and by $Q_h \coloneqq \varOmega_h \times I$ the discrete space-time domain. We assume that vertices on $\varGamma_h$ are points on $\varGamma$. 
\textcolor{C}{The interior vertices of $\mathcal{K}_h$ are denoted by $\{x_j\}_{j=1}^{N_h}$}. 
We combine the two discretization parameters $\tau$ and $h$ into the vector $\sigma = (\tau,h)$ and define the following discrete state and test spaces:
\textcolor{M1}{
\begin{align}
	\mathcal{Y}_{\sigma} 
	&\coloneqq 
	\Span \,\{ \SpaceNode_{x_j } \otimes \chi_k : 1\leq j \leq N_h \text{ and } 1\leq k\leq N_{\tau} \}, \label{eq:Ysigma}
	\\
	\cW_{\sigma} 
	&\coloneqq 
	\Span \, \{ \SpaceNode_{x_j } \otimes \TimeNode_{t_k} : 1\leq j \leq N_h \text{ and } 0\leq k\leq N_{\tau}-1 \}, \notag
\end{align}
}
\textcolor{Mn}{such that $w_{\sigma}(T) = 0$ for $w_{\sigma}\in \cW_{\sigma}$ is ensured.} Here, $(\SpaceNode_{x_1},\dotsc, \SpaceNode_{x_{N_h}})$ and \textcolor{Cn}{$(e_{t_0},\dotsc, e_{t_{N_\tau - 1}})$}
denote the nodal basis formed by continuous, piecewise linear functions on~$\varOmega_h$ and $\bar I$, respectively.
Moreover, $\chi_k$ denotes the indicator function of the time interval $I_k$.
We also define the space
\begin{align*}
	Y_h &\coloneqq \Span \,\{\SpaceNode_{x_j } : 1 \leq j \leq N_h\}, \label{eq:Yh}
\end{align*}
 
%
In order to set up the variational discrete state equation, we start by deriving a very weak formulation of \eqref{eq:PDE}, which will be discretized afterwards. By multiplication with $w \in W$, integration over the domain $Q$, and utilizing $w(x,T)=0$, we arrive at
\begin{align}
\textstyle
 \int_{Q} \bigparen{-y \, \partial_t w + \nabla y \nabla w} \, \dd x\, \dd t  
&= 
\textstyle
\int_{\bar{\varOmega}_c}{w(0) \, \dd u_{0}} + \int_{\bar{Q}_c}{w \, \dd u}.
\end{align} 
Inserting $y_{\sigma} \in \mathcal{Y}_{\sigma}$ and testing against all $w_{\sigma} \in \cW_{\sigma}$ yields the following variational discrete representation of the state equation: Find $y_{\sigma} \in \mathcal{Y}_{\sigma}$, such that 
\begin{equation}
\textstyle
 \int_{Q} \bigparen{-y_{\sigma} \, \partial_t w_{\sigma}+ \nabla y_{\sigma} \nabla w_{\sigma}} \, \dd x\, \dd t 
=  \int_{\bar{\varOmega}_c}{w_{\sigma}(0) \, \dd u_{0}}+\int_{\bar{Q}_c}{w_{\sigma} \, \dd u}  
\label{eq:dse2}
\end{equation}
holds for all $w_{\sigma} \in \cW_{\sigma}$.
This allows us to formulate the variational discrete problem
\begin{equation}
\min_{(u_{0},u) \in \mathcal{M}(\bar{\varOmega}_c) \times \mathcal{M}(\bar{Q}_c)} 
	J_{\sigma}(u_{0},u) 
	\ceq
	\tfrac{1}{q} \, \nnorm{y_{\sigma}(u_{0},u)-\DesiredState}_{L^q(Q_h)}^{q} 
	+ \alpha \, \nnorm{u}_{\mathcal{M}(\bar{Q}_c)}
	+ \beta \, \nnorm{u_{0}}_{\mathcal{M}(\bar{\varOmega}_c)}
\tag{$P_{\sigma}$},
\label{eq:Pvd}
\end{equation}
where $y_{\sigma}(u_{0},u)$ denotes the unique solution of \eqref{eq:dse2}.

%
%
%
We refrain from giving a detailed derivation for an optimality system and the sparsity structure for this problem as this would closely follow the procedure in the continuous setting (see Lemma~\ref{lem:optimalitycond} and Remark~\ref{sparsity}).
Instead, we focus on analyzing how the controls $(u_0,u)$ are implicitly discretized. First we define the following sets of indices:
\begin{equation*}
	\cI_{\sigma} \coloneqq \{(j,k) : (x_j,t_k) \in \bar{Q}_c \} 
	\qquad \textrm{and} \qquad 
	\cI_h \coloneqq \{j : x_j \in \bar{\varOmega}_c\}.
\end{equation*}
We may suppose that the space-time discretization is sufficiently fine so that $\cI_{\sigma}$ and $\cI_h$ are nonempty.
Utilizing these sets, we define the following discrete spaces:
\begin{align*}
	V_h \coloneqq \Span \, \{ \SpaceNode_{x_j }|_{\bar{\varOmega}_c} : j \in \cI_h \} 
	\quad
	\text{and}
	\quad
	\cV_{\sigma} \coloneqq \Span \, \{ (\SpaceNode_{x_j } \otimes \TimeNode_{t_k})|_{\bar{Q}_c} : (j,k) \in 	\cI_{\sigma}  \}.
\end{align*}
Notice that both spaces are spaces of continuous functions, i.e., we have
$V_h \subset \cC(\bar{\varOmega}_c)$ 
and
$\cV_{\sigma} \subset \cC(\bar{Q}_c)$.
The discrete version of the mapping $\Morrey$, decomposed into two parts, reads as follows:
\begin{alignat}{3}
\Morrey_{h} &: \cW_{\sigma} \rightarrow V_h, \quad &&\Morrey_{h}(w_{\sigma}) &&\coloneqq w_{\sigma}(0)|_{\bar{\varOmega}_c}, \label{eq:Morreyh}\\
\Morrey_{\sigma}&: \cW_{\sigma} \rightarrow \cV_{\sigma},\quad  &&\Morrey_{\sigma}(w_{\sigma}) &&\coloneqq w_{\sigma}|_{\bar{Q}_c}. \label{eq:Morreysigma}
\end{alignat}
We suppose that $\bar{\varOmega}_c$ is polygonal and that $\cK_{h}$ is an exact triangulations of $\bar{\varOmega}_c$, i.e., 
$\bar{\varOmega}_c = \bigcup_{K \in \cK_h : K \subset \bar{\varOmega}_c} K$ and similarly for $\bar{Q}_c$.


From the discrete optimality system (whose precise derivation has been omitted), we obtain:
\begin{alignat}{3}
\max_{j \in \cI_h} |\bar{w}_{j,0}| &= \nnorm {\bar{w}_{\sigma}(0)}_{\infty,\bar{\varOmega}_c} &&= \nnorm{\Morrey_{h}(\bar{w}_{\sigma}) }_{\infty,\bar{\varOmega}_c}   &&\leq  \beta, \label{eq:betanorm}\\
\max_{(j,k) \in \cI_{\sigma}} |\bar{w}_{j,k}| &= \nnorm{  \bar{w}_{\sigma} }_{\infty,\bar{Q}_c} &&= \nnorm{\Morrey_{\sigma}(\bar{w}_{\sigma}) }_{\infty,\bar{Q}_c}   &&\leq \alpha,\label{eq:alphanorm}
\end{alignat}
where the $\bar{w}_{j,k}$ denote the coefficients of the optimal adjoint variable $\bar{w}_{\sigma} \in \cW_{\sigma}$, i.e., 
$\bar{w}_{\sigma} = \sum_{j=1}^{N_h} \sum_{k=0}^{N_\tau-1} \bar{w}_{j,k} \, (\SpaceNode_{x_j } \otimes \TimeNode_{t_k})$.
Although we did not mention the mapping $\Morrey$ explicitly in the continuous optimality system,
we do so here for clarity.
As discrete sparsity structure, we obtain the following (see also Lemma~\ref{lem:optimalitycond})
\begin{align}
	\supp(\bar{u}_0^+) 
	&\subset \set{x \in \bar{\varOmega}_c : \bar{w}_{\sigma}(x,0)=-\beta }, 
	&\supp(\bar{u}_0^-) 
	&\subset \set{x \in \bar{\varOmega}_c : \bar{w}_{\sigma}(x,0)=+\beta }, 
	\notag
	\\
	\supp(\bar{u}^+) 
	&\subset \set{(x,t) \in \bar{Q}_c : \bar{w}_{\sigma}(x,t)=-\alpha }, 
	&\supp(\bar{u}^-) &\subset \set{(x,t) \in \bar{Q}_c : \bar{w}_{\sigma}(x,t)=+\alpha }. 
	\label{eq:DiscreteSparsityStructure}
\end{align}
By construction, the discrete adjoint state $\bar w_\sigma$ is piecewise linear, both in space and time.
Thus, $\varPhi_\sigma(\bar w_\sigma)$ and $\varPhi_h(\bar w_\sigma)$ attain their extremal values $\pm \alpha$ and $\pm \beta$ in the grid points contained in $\bar \varOmega_c$ and $\bar Q_c$, respectively.
Generically, $\varPhi_\sigma(\bar w_\sigma)$ and $\varPhi_h(\bar w_\sigma)$ attain their extrema \emph{only} in these grid points, in which case we have
\begin{equation*}
\supp(\bar{u}) \subset \set{(x_{j},t_k) : (j,k)\in \cI_{\sigma}} \qquad \textrm{and} \qquad \supp(\bar{u}_0) \subset \set{x_{j} : j \in \cI_h}.
\end{equation*}
This leads us in a natural way to the discrete control spaces
\begin{alignat}{2}
	U_h 
	&= \Span \, \{ \SpaceDelta_{x_j} : j \in \cI_h \} 
	&&
	\subset \cM(\bar{\varOmega }_c)
	\label{eq:Uh},
	\\
\mathcal{U}_{\sigma} 
	&= \Span \, \{ \SpaceDelta_{x_j} \otimes \TimeDelta_{t_k} : (j,k) \in \cI_{\sigma} \} 
	&&\subset \cM(\bar{Q}_c)   . 
	\label{eq:coeffvd}
\end{alignat}
Notice also that the natural pairings $\cM(\bar{\varOmega }_c) \times \cC(\bar{\varOmega }_c) \to \R$ and
$\cM(\bar{Q}_c) \times \cC(\bar Q_c) \to \R$
induce the \textcolor{M1}{dualities $V_h^* \cong U_h$ and $\cV_{\sigma}^* \cong \cU_{\sigma}$ in the discrete setting.}

Here we see the \textcolor{C}{effect} of the variational discretization concept:
\emph{The choice for the discretization of the test space induces a natural discretization for the controls.}

The following operators will be useful for the discussion of solutions to \eqref{eq:Pvd}: 
\begin{lemma}
	Let the linear operators $\varUpsilon_h$ and $\varPi_h$ be defined as below:
	\begin{alignat}{5}
	\varUpsilon_h &: \cM(\bar{\varOmega}_c) &&\rightarrow U_h &&\subset \cM(\bar{\varOmega}_c), \qquad &&\varUpsilon_h\, u_0 &&\coloneqq \textstyle \sum\limits_{j \in \cI_h} \updelta_{x_j} \int_{\bar{\varOmega}_c} \SpaceNode_{x_j } \dd u_0\notag 
	\\
	\varPi_h &: \cC(\bar{\varOmega}_c) &&\rightarrow V_h &&\subset \cC(\bar{\varOmega}_c), \qquad &&\varPi_h \,f_0 &&\coloneqq \textstyle \sum\limits_{j \in \cI_h} f_0(x_j)\, \SpaceNode_{x_j } \notag
	\end{alignat}
	Then for every $u_0 \in\cM(\bar{\varOmega}_c), f_0 \in \cC(\bar{\varOmega}_c)$ and $v_h \in V_h$ the following properties hold.
	\begin{alignat}{2}
	&\ninnerprod{ u_0, v_h } &&= \ninnerprod{ \varUpsilon_h \,u_0,v_h }, \label{eq:4.3}\\
	&\ninnerprod{ u_0, \varPi_h\, f_0 } &&= \ninnerprod{ \varUpsilon_h\, u_0, f_0 }, \label{eq:4.4}\\
	&\nnorm{\varUpsilon_h\, u_0}_{\mathcal{M}(\bar{\varOmega}_c)} &&\leq \nnorm{u_0}_{\mathcal{M}(\bar{\varOmega}_c)}, \label{eq:4.5}\\
	& \varUpsilon_h \,u_0 \stackrel{*}{\rightharpoonup} u_0 \; &&\in \; \mathcal{M}(\bar{\varOmega}_c) \quad \textrm{ and }\quad \nnorm{\varUpsilon_h\, u_0}_{\mathcal{M}(\bar{\varOmega}_c)} \stackrel{h \rightarrow 0}{\longrightarrow} \nnorm{u_0}_{\mathcal{M}(\bar{\varOmega}_c)} .\label{eq:4.6}
	\end{alignat}
\end{lemma}
These results follow directly from restricting \cite[Proposition 4.1.]{CasKun} to $\bar{\varOmega}_c \subset \varOmega$.

We derive also an analogous result for the space-time discrete spaces $\cU_{\sigma} $ and $\cV_{\sigma}$, which is similar to \cite[Proposition 4.2.]{CasKun}, but adjusted to our choice of spaces. The structure of the proof remains the same, only technical calculations are different.
\begin{lemma}
	Let the linear operators $\varUpsilon_{\sigma}$ and $\varPi_{\sigma}$ be defined as below:
	\begin{alignat}{5}
	\varUpsilon_{\sigma} &: \mathcal{M}(\bar{Q}_c) &&\rightarrow \mathcal{U}_{\sigma}  &&\subset \mathcal{M}(\bar{Q}_c), \qquad && \varUpsilon_{\sigma}\,u &&\coloneqq \textstyle \sum\limits_{(j,k)\in \cI_{\sigma}}   \SpaceDelta_{x_j} \otimes \TimeDelta_{t_k} \, \int_{\bar{Q}_c}{\SpaceNode_{x_j } \otimes \TimeNode_{t_k} \, \dd u}    \notag \\
	\varPi_{\sigma} &: \cC(\bar{Q}_c) &&\rightarrow \cV_{\sigma} &&\subset \cC(\bar{Q}_c) , \qquad && \varPi_{\sigma}\, f &&\coloneqq \textstyle \sum\limits_{(j,k)\in\cI_{\sigma}}{{ f(x_j,t_k) \, (\SpaceNode_{x_j} \otimes \TimeNode_{t_k} ) }} \notag
	\end{alignat}
	Then for every $u \in\cM(\bar{Q}_c), f \in \cC(\bar{Q}_c)$ and  $v_{\sigma} \in \mathcal{V}_{\sigma}$ the following properties hold.
	\begin{alignat}{2}
	&\ninnerprod{ u, v_{\sigma} } 
	&&= \ninnerprod{ \varUpsilon_{\sigma}\, u ,v_{\sigma}}, \label{eq:4.8}\\
	&\ninnerprod{ u, \varPi_{\sigma}\, f } 
	&&= \ninnerprod{ \varUpsilon_{\sigma}\, u, f }, \label{eq:4.9}\\
	&\nnorm{\varUpsilon_{\sigma} \,u}_{\mathcal{M}(\bar{Q}_c)} 
	&&\leq \nnorm{u}_{\mathcal{M}(\bar{Q}_c)}, \label{eq:4.10}\\
	& \varUpsilon_{\sigma} \,u \stackrel{*}{\rightharpoonup} u \; 
	&&\in \; \mathcal{M}(\bar{Q}_c) 
	\quad \textrm{ and }\quad 
	\nnorm{\varUpsilon_{\sigma}\, u}_{\mathcal{M}(\bar{Q}_c)} \stackrel{\nabs{\sigma} \rightarrow 0}{\longrightarrow} \nnorm{u}_{\mathcal{M}(\bar{Q}_c)} \label{eq:4.11}
	\end{alignat}
\end{lemma}

\textcolor{M1}{
	\begin{proof}
		We verify \eqref{eq:4.8} by the following calculation: 
		\begin{alignat}{2}
		\left\langle u, v_{\sigma}\right\rangle &= \textstyle\int_{\bar{Q}_c}{ \sum\limits_{(j,k)\in \cI_{\sigma}} {v_{j,k}\, (\SpaceNode_{x_j }\otimes \TimeNode_{t_k})  } \, \dd u } 
		\notag\\
		&= \textstyle\sum\limits_{(j,k)\in \cI_{\sigma}} v_{\sigma}(x_j,t_k)  \int_{\bar{Q}_c} \SpaceNode_{x_j }\otimes \TimeNode_{t_k} \, \dd u \notag\\
		&= \textstyle\sum\limits_{(j,k)\in \cI_{\sigma}} \langle \SpaceDelta_{x_j} \otimes \TimeDelta_{t_k}, v_{\sigma}\rangle  \int_{\bar{Q}_c} \SpaceNode_{x_j }\otimes \TimeNode_{t_k} \, \dd u \; &&= \left\langle \varUpsilon_{\sigma} \, u, v_{\sigma}\right\rangle \notag
		\end{alignat}
		Equation \eqref{eq:4.9} can be seen by another short calculation:
		\begin{equation*}
		\left\langle u, \varPi_{\sigma}\, f\right\rangle = 
		\textstyle
		\int_{\bar{Q}_c} \textstyle\sum\limits_{(j,k)\in \cI_{\sigma}}f(x_j,t_k) \, (\SpaceNode_{x_j} \otimes \TimeNode_{t_k} ) \, \dd u = \textstyle\sum\limits_{(j,k)\in \cI_{\sigma}} \langle \SpaceDelta_{x_j} \otimes \TimeDelta_{t_k}, v_{\sigma}\rangle  \int_{\bar{Q}_c} \SpaceNode_{x_j }\otimes \TimeNode_{t_k} \, \dd u 
		\end{equation*}
		Inequality \eqref{eq:4.10} is obtained as follows:
		\begin{align*}
		\left\|\varUpsilon_{\sigma}\,u\right\|_{\mathcal{M}(\bar{Q}_c)} 
		&= 
		\big\| 
			\textstyle \sum\limits_{(j,k)\in \cI_{\sigma}}   \SpaceDelta_{x_j} \otimes \TimeDelta_{t_k} \, \int_{\bar{Q}_c}{\SpaceNode_{x_j } \otimes \TimeNode_{t_k} \, \dd u}  \, 
		\big\|_{\mathcal{M}(\bar{Q}_c)} \\
		&\leq  
		\textstyle\sum\limits_{(j,k)\in \cI_{\sigma}} \underbrace{\|\SpaceDelta_{x_j } \otimes \TimeDelta_{t_k} \|_{\mathcal{M}(\bar{Q}_c)}}_{=1}  \int_{\bar{Q}_c}{\SpaceNode_{x_j } \otimes \TimeNode_{t_k} \, \dd |u|}
		\leq \int_{\bar{Q}_c} \dd |u|  = \left\|u\right\|_{\mathcal{M}(\bar{Q}_c)}
		\end{align*}
		We deduce that there exists a subsequence, denoted in the same way, such that 
		\begin{equation*}
		\varUpsilon_{\sigma}\,u \stackrel{*}{\rightharpoonup} \bar{u} \; \in \mathcal{M}(\bar{Q}_c) \; \textrm{as} \; \left|\sigma\right|\rightarrow 0.
		\end{equation*}
		For any $f \in \cC(\bar{Q}_c)$, it holds that $\varPi_{\sigma}\,f \rightarrow f \in \cC(\bar{Q}_c)$ as $\left|\sigma\right|\rightarrow 0$. We derive
		\begin{equation*}
		\left\langle \bar{u},f\right\rangle = \lim_{\left|\sigma\right|\rightarrow 0} \left\langle \varUpsilon_{\sigma}\, u, f\right\rangle \stackrel{\eqref{eq:4.9}}{=} \lim_{\left|\sigma\right|\rightarrow 0} \left\langle  u,\varPi_{\sigma}\, f\right\rangle = \left\langle u,f\right\rangle.
		\end{equation*}
		This shows $\bar{u} = u$ and $\varUpsilon_{\sigma}\,u \stackrel{*}{\rightharpoonup} u$ for the whole sequence and hence
		\begin{equation*}
		\left\|u\right\|_{\mathcal{M}(\bar{Q}_c)} \leq \liminf_{\left|\sigma\right|\rightarrow 0} \left\|\varUpsilon_{\sigma}\,u\right\|_{\mathcal{M}(\bar{Q}_c)} \stackrel{\eqref{eq:4.10}}{\leq} \left\|u\right\|_{\mathcal{M}(\bar{Q}_c)} .
		\end{equation*}  
	\end{proof}
}

Next, we observe that $J_{\sigma}$ is convex, but not strictly convex. \textcolor{Mn}{In the continuous setting,} the strict convexity of $J$ was caused by the norm $\nnorm{\cdot}_{L^q(Q)}$ for $q>1$ and the injectivity of $\adjheat^{-*}\Morrey^*$. \textcolor{C}{Here we have the discrete operator $\adjheat_{\sigma}^*: \cY_{\sigma} \rightarrow \cW_{\sigma}^* $, defined as $\langle \adjheat_{\sigma}^* y_{\sigma}, w_{\sigma}\rangle \coloneqq \int_{Q}(- y_{\sigma} \, \partial_t w_{\sigma} +  \nabla y_{\sigma}\nabla w_{\sigma}) \, \dd x \, \dd t$, with $\adjheat_{\sigma}: \cW_{\sigma} \rightarrow \cY_{\sigma}^*$. We can rewrite the discrete state equation \eqref{eq:dse2} as 
\begin{equation}
 y_{\sigma}(u_0,u) = \adjheat_{\sigma}^{-*}( \Morrey_{h}^* \varUpsilon_h u_0 + \Morrey_{\sigma}^* \varUpsilon_{\sigma} u).
 \label{eq:ysigma(u)}
\end{equation}
}
The mapping $(u_0,u) \mapsto y_{\sigma}(u_0,u)$ is in general not injective, hence the uniqueness of the solution cannot be concluded. In the implicitly discrete setting however, we can prove uniqueness similarly as done in \cite[Section 4.3.]{CasasClasonKunisch}.
%
\begin{theorem}
	\label{thm:exandproj}
	The problem \eqref{eq:Pvd} has at least one solution in $\mathcal{M}(\bar{\varOmega}_c) \times \mathcal{M}(\bar{Q}_c) $ and there exists a unique solution $\nparen{ \bar{u}_{0,h},\bar{u}_{\sigma}} \in U_h \times \mathcal{U}_{\sigma} $. Furthermore we know for every solution $\nparen{\hat{u}_{0}, \hat{u}}\in \mathcal{M}(\bar{\varOmega}_c) \times \mathcal{M}(\bar{Q}_c) $ of \eqref{eq:Pvd} that $\nparen{ \varUpsilon_h\, \hat{u}_{0},\varUpsilon_{\sigma} \, \hat{u} }  = \nparen{ \bar{u}_{0,h},\bar{u}_{\sigma} }$.
\end{theorem}

\begin{proof}
	The existence of solutions can be derived as in the proof of Theorem~\ref{thm:PisUniquelySolvable} because the control domain is still continuous. 
	
	Let $\nparen{ \hat{u}_{0}, \hat{u} } \in \mathcal{M}(\bar{\varOmega}_c) \times \mathcal{M}(\bar{Q}_c) $ be a solution of problem \eqref{eq:Pvd} and let $\nparen{ \bar{u}_{0,h},\bar{u}_{\sigma}} \coloneqq \nparen{ \varUpsilon_h\, \hat{u}_{0},\varUpsilon_{\sigma} \, \hat{u} } \in U_h \times \mathcal{U}_{\sigma} $.  
	We deduce from \eqref{eq:4.3} and \eqref{eq:4.8} that 
	\begin{equation}
	\textstyle
	y_{\sigma}(u_0,u) = y_{\sigma}(\varUpsilon_h u_0, \varUpsilon_{\sigma} u)
	\quad 
	\text{for all $(u_0,u) \in \mathcal{M}(\bar{\varOmega}_c) \times \mathcal{M}(\bar{Q}_c)$.}
	\label{eq:ygleich}
	\end{equation}
	Additionally, \eqref{eq:4.5} and \eqref{eq:4.10} deliver $\lVert \bar{u}_{0,h} \rVert_{\mathcal{M}(\bar{\varOmega}_c)} \leq \nnorm{\hat{u}_{0}}_{\mathcal{M}(\bar{\varOmega}_c)}$ and $\nnorm{\bar{u}_{ \sigma}}_{\mathcal{M}(\bar{Q}_c)} \leq \nnorm{\hat{u}}_{\mathcal{M}(\bar{Q}_c)}$. 
	Combining these properties, we deduce $J_{\sigma} \nparen{ \bar{u}_{0,h},\bar{u}_{\sigma} } \leq J_{\sigma} \nparen{ \hat{u}_{0}, \hat{u} }$. 
	This validates the existence of solutions in the discrete space $U_h \times \mathcal{U}_{\operatorname{vd}} $.

	\textcolor{M1}{The operators $\varUpsilon_h$ and $\varUpsilon_{\sigma}$ act as identities on $U_h$ and $\cU_{\sigma}$. 
	Furthermore, the operator $(\Morrey_{h} , \Morrey_{\sigma})^*: U_h \times \cU_{\sigma} \rightarrow \cW_{\sigma}^*$ is injective and we also know that $\dim(\cY_{\sigma} ) = \dim(\cW_{\sigma}^*)$. 
	Hence we deduce the injectivity of $\nparen{ u_0,u } \mapsto y_{\sigma} \nparen{ u_0,u }$ for discrete controls $\nparen{ u_0,u } \in U_h \times \mathcal{U}_{\sigma}$. }
	Now strict convexity of $J_{\sigma}$ on $U_h \times \mathcal{U}_{\operatorname{vd}}$ follows
	from $q>1$. Consequently, problem \eqref{eq:Pvd} has a unique discrete solution $\nparen{ \bar{u}_{0,h},\bar{u}_{\sigma} } \in U_h \times \mathcal{U}_{\sigma} $.
	
	For \textcolor{Mn}{every} solution $(\hat{u}_{0}, \hat{u})$ of \eqref{eq:Pvd},
	the projection $\nparen{ \varUpsilon_h\, \hat{u}_{0},\varUpsilon_{\sigma} \, \hat{u} }$
	is a discrete solution. 
	Moreover, there exists only one discrete solution. 
	So we deduce that all projections must coincide, showing the second statement of the theorem.
\end{proof}

Since all projections of solutions  of \eqref{eq:Pvd} yield the unique discrete solution $\nparen{ \bar{u}_{0,h},\bar{u}_{\sigma} }$, it suffices to analyze the convergence properties of $\nparen{ \bar{u}_{0,h},\bar{u}_{\sigma} }$ for $\nabs{\sigma} \to 0$. 
Furthermore we may find solutions of \eqref{eq:Pvd} numerically, by restricting the control space to $U_h \times \cU_{\sigma}$.

\bigskip
%
%
%
%
We can now prove the convergence result formulated in Theorem~\ref{thm:DiscVarConvergence} along the lines of the proof of \cite[Theorem 4.3.]{CasKun}. 
\begin{proof} \label{convergenceproof}
	
	\textcolor{M2}{Observe that $J_{\sigma}(\bar{u}_{0,h},\bar{u}_{\sigma}) \leq J_{\sigma}(0,0) = \tfrac{1}{q}\|\DesiredState\|_{L^q(Q_h)}^q$. This implies that the norms $\|\bar{y}_{\sigma}\|_{L^q(Q_h)}$, $\|\bar{u}_{0,h}\|_{\cM \left( \bar{\varOmega }_c \right) },$ and $\|\bar{u}_{\sigma}\|_{\cM \left(\bar{Q}_c \right)}$ are uniformly bounded for all $\sigma$.} Now, let $\{\sigma_n\}_n$ be a sequence with $|\sigma_n| \rightarrow 0$. Boundedness in norm implies that there exists a subsequence $\{\sigma_{n_k}\}_k$, such that the following holds true for $k \rightarrow \infty$:
	\begin{equation}
		\nparen{\bar{u}_{0,h_{n_k}}, \bar{u}_{\sigma_{n_k}}} \stackrel{*}{\rightharpoonup} (\tilde{u}_0,\tilde{u}) 
		\text{ in } 
		\mathcal{M}(\bar{\varOmega}_c) \times \mathcal{M}(\bar{Q}_c) 
		\quad \text{and} \quad  
		\bar{y}_{\sigma_{n_k}} \rightharpoonup \tilde{y} \text{ in } L^q(Q). \label{eq:4.17}
	\end{equation}
	We will split the proof into three steps.

	
	%
	\textbf{I -} \textit{$ \tilde{y}$ is the solution of \eqref{eq:PDE} corresponding to $(\tilde{u}_0,\tilde{u})$, i.e., $\adjheat^* \tilde{y} = \Morrey^* (\tilde{u}_0,\tilde{u})$.}
	\newline
	\textcolor{M1}{By constructing a suitable Friedrichs smoothing operators and by using standard results in approximation theory (e.g., cubic spline interpolation, see \cite[Theorem 1]{AhlbergNilson63}),
	one realizes that $\set{\psi \in \cC^2(\bar I;\R) : \psi(T) = 0 } \otimes (\cC^{2}(\varOmega) \cap W^{1,p}_0 (\varOmega))$ is dense in $W$.} Consequently, it is sufficient to test the smooth state equation \eqref{eq:2.1} against $w = \varphi \otimes \psi$ with $\psi \in \cC^2(\bar I;\R)$ satisfying $\psi(T) = 0$ and $\varphi \in \cC^{2}(\varOmega) \cap W^{1,p}_0 (\varOmega)$. 
	Let $\varphi$ be approximated by $\varphi_h \in Y_h$, such that
	\begin{equation}
		\textstyle 
		\int_\varOmega \nabla\varphi_h \nabla z_h \, \dd x = \int_\varOmega \nabla\varphi \nabla z_h \, \dd x \quad \text{for all  $w_h \in Y_h$} 
		\quad \text{and} \quad 	
		\nnorm{\varphi - \varphi_h}_{\cC(\bar{\varOmega})} \converges{h\to 0} 0.
		\label{eq:4.18}
	\end{equation}
	\textcolor{M1}{Indeed, one has the error estimate $\nnorm{\varphi - \varphi_h}_{L^\infty(\varOmega)}
		\leq C \ h^2 \, \log(1/h)\, \nnorm{\varphi}_{W^{2,\infty}(\varOmega)}$ for the Ritz projection.
		For details see Corollary~2 and Remark~4 in \cite{MR568857}.
		}
	Moreover, let $\psi_\tau = \sum_k \psi(t_k) \, \TimeNode_{t_k}$ be the continuous, piecewise linear interpolation of $\psi$ on the time grid and put $w_{\sigma} = \varphi_h \otimes \psi_{\tau} \in \cW_{\sigma}$.
	By construction, we have $w_{\sigma} \rightarrow w \in \cC(\bar{Q})$ for $|\sigma|\rightarrow 0$. Furthermore, $\partial_t w_{\sigma} \rightarrow \partial_t w \in L^p(Q)$\textcolor{Mn}{, where $\frac{1}{q}+\frac{1}{p} =1$,} for $|\sigma|\rightarrow 0$ since
	\begin{align*}
		\nnorm{\partial_t w - \partial_t w_\sigma}_{L^p(Q)}
		&\leq
		\nnorm{(\varphi - \varphi_h) \otimes \psi'}_{L^p(Q)}			
		+
		\nnorm{\varphi_h \otimes (\psi' - \psi_\tau')}_{L^p(Q)}			
		\\
		&=
		\nnorm{\varphi - \varphi_h}_{L^p(\varOmega)} \, \nnorm{\psi'}_{L^p(I)}			
		+
		\norm{\varphi_h}_{L^p(\varOmega)} \, \nnorm{\psi' - \psi_\tau'}_{L^p(I)}						
		\\
		&\leq
		\nabs{\varOmega}^{1/p} 
		\,
		\bigparen{
		\nnorm{\varphi - \varphi_h}_{\cC(\bar{\varOmega})} \, \nnorm{\psi'}_{L^p(I)}	
		+
		\norm{\varphi_h}_{\cC(\bar{\varOmega})} \, \nnorm{\psi' - \psi_\tau'}_{L^p(I)}					
		}
		.	
	\end{align*}	
\textcolor{M1}{ 	
	We have $\nnorm{\varphi - \varphi_h}_{\cC(\bar{\varOmega})} \converges{h\to 0} 0$ from \eqref{eq:4.18} 
	and $\nnorm{\psi' - \psi_\tau'}_{L^p(I)} \leq C \, h \nnorm{\psi''}_{\cC(\bar I)} \converges{\tau \to 0} 0$ can be confirmed by splitting $\bar{I}$ into its subintervals $I_k$, integrating and using $\psi(t_k) = \psi_{\tau}(t_k)$ for all $k \in \{1,\ldots,\N_{\tau}\}$.}
	Testing \eqref{eq:dse2} against this $w_{\sigma}$, we obtain
	\begin{equation}
	\langle \adjheat_{\sigma}^* \bar{y}_{\sigma}, w_{\sigma}\rangle =\textstyle \int_{\bar{\varOmega}_c} w_{\sigma}(0) \, \dd \bar{u}_{0,h} + \int_{\bar{Q}_c} w_{\sigma} \, \dd \bar{u}_{\sigma}.
	\label{eq:dseproduct}
	\end{equation}
	On the right hand side, we can perform the limit directly:
	\begin{equation*}
	\textstyle  \int_{\bar{\varOmega}_c} w_\sigma(0) \, \dd \bar{u}_{0,h}+ \int_{\bar{Q}_c} w_\sigma \, \dd \bar{u}_\sigma
	\converges{\abs{\sigma} \to 0} \int_{\bar{\varOmega}_c}  w(0) \, \dd \tilde u_0 + \int_{\bar{Q}_c} w \,  \dd \tilde  u.
	\end{equation*}
	The left hand side of \eqref{eq:dseproduct} can be expanded to
	\begin{equation}
	\langle \adjheat_{\sigma}^* \bar{y}_{\sigma}, w_{\sigma}\rangle = - \textstyle\int_Q \bar{y}_{\sigma} (\varphi_h \otimes \psi_\tau') \, \dd x \, \dd t + \int_Q \nabla \bar{y}_{\sigma} \nabla (\varphi_h \otimes \psi_{\tau}) \, \dd x \, \dd t. 
	\end{equation}
	Applying the very definition of $\varphi_h$ and integration by parts, we observe that 
	\begin{align*}
		\textstyle
		\int_Q \nabla \bar{y}_{\sigma} \, \nabla (\varphi_h \otimes  \psi_\tau) \,\dd x\, \dd t
		&=
		\textstyle
		- \!\int_Q \bar{y}_{\sigma} \, (\Delta \varphi \otimes \psi_\tau ) \, \dd x \, \dd t
		\converges{\abs{\sigma} \to 0}
		- \!\int_Q \tilde{y} \, \Delta  w \, \dd x \, \dd t.
	\end{align*}
	Along with $ (\varphi_h \otimes \psi_\tau') = \partial_t w_{\sigma}$ and $\textstyle -\int_Q \bar{y}_{\sigma} \, \partial_t w_{\sigma}\, \dd x\, \dd t \to  -\int_Q \tilde{y}\, \partial_t w\, \dd x\, \dd t$,  
	this implies that $\langle \adjheat_{\sigma}^* \bar{y}_{\sigma}, w_{\sigma}\rangle \to  \langle \adjheat^* \tilde{y}, w \rangle$ for all tensor product functions $w = \varphi \otimes \psi$, Thus, we deduce $\adjheat^* \tilde{y} = \Morrey^* (\tilde{u}_0,\tilde{u})$ from \eqref{eq:dseproduct}.
	%
	%
	%
	\newline
	\textcolor{M2}{
	\textbf{II -} \textit{$(\tilde{u}_0,\tilde{u})$ coincides with the unique solution $(\bar{u}_0,\bar{u})$ of \eqref{eq:Pvd}} that lies in $U_h \times \cU_\sigma$ 
	\newline
	In order to prove this, it suffices to show $J_{\sigma}(\tilde{u}_0,\tilde{u}) \leq J_{\sigma}(\bar{u}_0,\bar{u}) $. 
	\textcolor{Mn}{Recall that we identified $\cM(\bar{\varOmega }_c)$ and $\cM(\bar{Q}_c)$ with $\{u_0 \in \cM (\varOmega) \colon \supp(u_0) \subset \bar{\varOmega}_c \}$ and $\{ u \in \cM(Q) \colon \supp(u) \subset \bar{Q}_c \}$, respectively. In this sense, the sets $\{ f_0 \in \cC^{\infty}(\varOmega) \colon \supp(f_0) \subset \bar{\varOmega}_c \}$ and $ \{ f \in \cC^{\infty}(Q) \colon \supp(f) \subset \bar{Q}_c \} $ are dense in $\cM(\bar{\varOmega}_c)$ and $\cM(\bar{Q}_c)$ with respect to the sequential weak-* topology.\footnote{This can be seen by utilizing that $\bar \varOmega_c$ and $\bar Q_c$ have to satisfy certain uniform cone conditions (because they are Lipschitz domains, see \cite[Paragraph~4.8]{MR2424078}) and by convolution against suitable Friedrichs mollifiers that are compactly supported in the interior of finite, convex cones. Notice also that such convolutions do not increase the $\cM$-norms.} 
	}
	 Consequently, we may pick a specific minimizing sequence $(u_{0,m},u_m) = (f_{0,m} \, \dd \, x, f_m \, \dd \, x \otimes \dd \, t)$, where $f_{0,m} \in \cC^\infty(\varOmega)$ and $f_m \in \cC^\infty(Q)$ satisfy $\supp(f_{0,m}) \subset \varOmega_c$ and $\supp(f_m) \subset Q_c$. Then the states $y_m \coloneqq \adjheat^{-*}\Morrey^* (u_{0,m},u_m)$ are solutions of the heat equations
\begin{equation*}
	\left\{
	\begin{aligned}
		\partial_t y_m - \Delta y_m &=f_m, &&\text{in $Q$,}
		\\
		y_m(x,0) &=f_{0,m}, &&\text{in $\varOmega$,
		}\\
		y_m(x,t) &= 0, &&\text{on $\varSigma$.}
    \end{aligned}
    \right.
\end{equation*}
	It follows now from maximal regularity (recall that $\varOmega$ is now assumed to be of class $\cC^{1,1}$), that $y_m \in W^{2,1}_r(Q)$ for all $2 \leq r < \infty$.	
	Thus the finite element discretizations of the states converge to $y_m$ in $L^2(Q)$ and thus also in $L^q(Q)$. 
	More precisely, we have for each fixed $m$ that
	\begin{equation*}
	\lim_{\abs{\sigma}\rightarrow 0} \| \adjheat_{\sigma}^{-*} (\Morrey_{h}^* \oplus \Morrey_{\sigma}^*)(u_{0,m},u_m) -y_m\|_{L^q(Q)} =0.
	\end{equation*}
	We choose a suitable subsequence of $\{ \sigma_{n_{k}}\}_{k\in\N}$ as follows: 
	We put $k_1 \coloneqq 1$ and pick $k_m \geq k_{m-1}$ recursively  such that 
	\begin{equation*}
	\| \adjheat_{\sigma_{n_{k_m}}}^{-*} (\Morrey_{h_{n_{k_m}}}^* \oplus \Morrey_{\sigma_{n_{k_m}}}^*)(u_{0,m},u_m) -y_m\|_{L^q(Q)}^q \leq \tfrac{1}{m}
	\quad
	\text{for each $m \geq 2$.}
	\end{equation*}
	Now using 
	the projection properties \eqref{eq:4.5},\eqref{eq:4.10} and \eqref{eq:ygleich}  in combination with the above, we obtain
	\begin{align*}
	&J_{\sigma}(\bar{u}_{0,h_{n_{k_m}}},\bar{u}_{\sigma_{n_{k_m}}}) 
	\leq J_{\sigma} (\Upsilon_{h_{n_{k_m}}}u_{0,m}, \Upsilon_{\sigma_{n_{k_m}}}u_m)
	\\
	&\leq 
	\tfrac{1}{q} \| \adjheat_{\sigma_{n_{k_m}}}^{-*}(\Morrey_{h_{n_{k_m}}}^* \oplus \Morrey_{\sigma_{n_{k_m}}}^*)(u_{0,m},u_m) -y_m + y_m - \DesiredState \|_{L^q(Q)}^q 
	+ 
	\alpha \|u_m\|_{\cM(\bar{Q}_c)} + \beta \|u_{0,m}\|_{\cM(\bar{\varOmega}_c)}
	\\
	&
	\leq \tfrac{1}{q} \big( \tfrac{1}{m} + \|y_m - \DesiredState\|_{L^q(Q)}^q  \big) + 
	\alpha \|u_m\|_{\cM(\bar{Q}_c)} + \beta \|u_{0,m}\|_{\cM(\bar{\varOmega}_c)}
	.
	\end{align*}
	\textcolor{M1}{Next, we use the weakly lower semicontinuity of $J$ in combination with \eqref{eq:4.17},} then apply $\liminf_{m \rightarrow \infty }$ to both sides of the above inequality, and finally use the facts that $(u_{0,m},u_m)$ is a minimizing sequence and that $(\bar{u}_0,\bar{u})$ solves problem \eqref{eq:P}:
	\begin{align*}
	J(\tilde{u}_0,\tilde{u}) & \leq  \liminf_{m \rightarrow \infty } J_{\sigma}(\bar{u}_{0,h_{n_{k_m}}},\bar{u}_{\sigma_{n_{k_m}}}) \\
	& \leq
	\liminf_{m \rightarrow \infty } \tfrac{1}{q} \textcolor{Mn}{\left( \left( \tfrac{1}{m} + \|y_m - \DesiredState\|_{L^q(Q)}^q  \right) + 
	\alpha \|u_m\|_{\cM(\bar{Q}_c)} + \beta \|u_{0,m}\|_{\cM(\bar{\varOmega}_c)} \right)}\\
	&\textcolor{M1}{\leq \limsup_{m \rightarrow \infty }J_{\sigma}(u_{0,m},u_m)}\\
	&\leq \tfrac{1}{q} \|\bar{y}-\DesiredState\|_{L^q(Q)}^q + \alpha \| \bar{u} \|_{\cM(\bar{Q}_c)} + \beta \|\bar{u}_{0}\|_{\cM(\bar{\varOmega}_c)} = J(\bar{u}_0,\bar{u})
	.
	\end{align*}
    }
	
	\textbf{III -} \textit{proof of \eqref{eq:conv1}, \eqref{eq:conv2} and \eqref{eq:conv3} }
	\newline
	\textcolor{M1}{Altogether, we know that every sequence $\{\sigma_n\}_n$ with $\abs{\sigma_n}\rightarrow 0$ has a subsequence $\{\sigma_{n_{k_m}}\}_m$, such that
	\begin{equation*}
	\nparen{\bar{u}_{0,h_{n_{k_m}}}, \bar{u}_{\sigma_{n_{k_m}}}} \stackrel{*}{\rightharpoonup} (\bar{u}_0,\bar{u}) \in \mathcal{M}(\bar{\varOmega}_c) \times \mathcal{M}(\bar{Q}_c) \quad \textrm{and} \quad  \bar{y}_{\sigma_{n_{k_m}}} \rightharpoonup \bar{y} \in L^q(Q)
	\quad \text{for} \quad m \rightarrow \infty.
	\end{equation*}
	Since the limits are always the same \textcolor{Mn}{and $(\bar u_0, \bar u )$ and $\bar y$ are unique}, this implies already that
	\begin{equation*}
	\nparen{\bar{u}_{0,h}, \bar{u}_{\sigma}} \stackrel{*}{\rightharpoonup} (\bar{u}_0,\bar{u}) \in \mathcal{M}(\bar{\varOmega}_c) \times \mathcal{M}(\bar{Q}_c) \quad \textrm{and} \quad  \bar{y}_{\sigma} \rightharpoonup \bar{y} \in L^q(Q) \quad \textrm{for } \abs{\sigma} \rightarrow 0.
	\end{equation*}
	This shows \eqref{eq:conv2}.} Next, we can calculate
	\begin{align*}
		\tfrac{1}{q} \nnorm{\bar{y} - \DesiredState }_{L^q(Q)}^q 
		&
		\stackrel{\phantom{\eqref{eq:4.17}}}{\leq}
		\liminf_{\nabs{\sigma}\rightarrow 0} \tfrac{1}{q} \nnorm{\bar{y}_{\sigma} - \DesiredState }_{L^q(Q)}^q \textcolor{M1}{\leq \limsup_{\nabs{\sigma}\rightarrow 0} \tfrac{1}{q} \nnorm{\bar{y}_{\sigma} - \DesiredState }_{L^q(Q)}^q}
		\\
		&
		\stackrel{\phantom{\eqref{eq:4.17}}}{\leq}
		\limsup_{\nabs{\sigma}\rightarrow 0} \bigparen{ J_{\sigma}(\bar{u}_{0,h},\bar{u}_{\sigma}) 
		- \alpha\, \nnorm{\bar{u}_{\sigma}}_{\mathcal{M}(Q)} 
		- \beta \, \nnorm{\bar{u}_{0,h}}_{\mathcal{M}(\varOmega)} }
		\\
		& 
		\stackrel{\phantom{\eqref{eq:4.17}}}{\leq}
		\limsup_{\nabs{\sigma}\rightarrow 0} J_{\sigma}(\bar{u}_{0,h},\bar{u}_{\sigma}) 
		- \liminf_{\nabs{\sigma}\rightarrow 0} \bigparen{\alpha \, \nnorm{\bar{u}_{\sigma}}_{\mathcal{M}(Q)} + \beta \, \lVert \bar{u}_{0,h} \rVert_{\mathcal{M}(\varOmega)} }
		\\
		&
		\stackrel{\phantom{\eqref{eq:4.17}}}{=} 
		J ( \bar{u}_0,\bar{u}) -\liminf_{\nabs{\sigma}\rightarrow 0} \bigparen{\alpha \nnorm{\bar{u}_{\sigma}}_{\mathcal{M}(Q)} + \beta \lVert \bar{u}_{0,h} \rVert_{\mathcal{M}(\varOmega)} }
		\\
		&\stackrel{\eqref{eq:4.17}}{\leq} J ( \bar{u}_0,\bar{u}) - \bigparen{\alpha \nnorm{\bar{u}}_{\mathcal{M}(Q)} + \beta \nnorm{\bar{u}}_{\mathcal{M}(\varOmega)} }
		\\
		&
		\stackrel{\phantom{\eqref{eq:4.17}}}{\leq}
		\tfrac{1}{q} \nnorm{\bar{y} - \DesiredState }_{L^q(Q)}^q.
	\end{align*}
	\textcolor{M1}{Because of $1<q<\infty$, the space $L^q(Q)$ is an uniformly convex Banach space; thus weak convergence together with the convergence of the norms implies strong convergence (see \cite[Proposition 3.32]{MR2759829}).
	This shows \eqref{eq:conv1}.} 
	In a similar way we can prove the first part of \eqref{eq:conv3}
		\begin{align*}
		\alpha \,\nnorm{\bar{u}}_{\mathcal{M}(Q)} 
		&
		\stackrel{\eqref{eq:4.17}}{\leq} 
		\liminf_{\nabs{\sigma}\rightarrow 0} \alpha \, \nnorm{\bar{u}_{\sigma}}_{\mathcal{M}(Q)}
		\leq 
		\textcolor{M1}{\limsup_{\nabs{\sigma}\rightarrow 0} \alpha \, \nnorm{\bar{u}_{\sigma}}_{\mathcal{M}(Q)}}
		\\
		&
		\stackrel{\phantom{\eqref{eq:4.17}}}{\leq}
		\limsup_{\nabs{\sigma}\rightarrow 0} \bigparen{ J_{\sigma}(\bar{u}_{0,h},\bar{u}_{\sigma}) -\tfrac{1}{q} \nnorm{\bar{y}_{\sigma} - \DesiredState }_{L^q(Q)}^q - \beta \, \nnorm{\bar{u}_{0,h}}_{\mathcal{M}(\varOmega)} }
		\\
		&
		\,\stackrel{\eqref{eq:conv1}}{\textcolor{M1}{\leq}}
		J ( \bar{u}_0,\bar{u}) - \tfrac{1}{q} \nnorm{\bar{y}- \DesiredState }_{L^q(Q)}^q - \liminf_{\nabs{\sigma}\rightarrow 0} \beta \, \nnorm{\bar{u}_{0,h}}_{\mathcal{M}(\varOmega)}
		\\
		&
		\stackrel{\eqref{eq:4.17}}{\leq} J ( \bar{u}_0,\bar{u}) - \tfrac{1}{q} \nnorm{\bar{y}- \DesiredState }_{L^q(Q)}^q - \beta \nnorm{\bar{u}_{0}}_{\mathcal{M}(\varOmega)}
		\\
		&
		\stackrel{\phantom{\eqref{eq:4.17}}}{=} 
		\alpha \, \nnorm{\bar{u}}_{\mathcal{M}(Q)}.
		\end{align*}
		Finally, the second part of \eqref{eq:conv3} follows directly from $\lim_{\nabs{\sigma}\rightarrow 0} J_{\sigma}(\bar{u}_{0,h},\bar{u}_{\sigma}) =J ( \bar{u}_0,\bar{u})$ and the fact that we already showed the convergence of the other two terms.

\end{proof}
%

Now we discretize \eqref{eq:Pdual} with $w_{\sigma } \in \cW_{\sigma}$ and equivalently reformulate the problem in the following way:
\begin{alignat}{2}
&\min_{w_{\sigma}\in \cW_{\sigma}} &&K_{\sigma}(w_{\sigma}) \coloneqq \tfrac{1}{p}\|\adjheat_{\sigma} w_{\sigma}\|_{L^p(Q_h)}^p +  \langle  \adjheat_{\sigma} w_{\sigma}, \textcolor{C}{\DesiredState}\rangle_{L^p(Q_h),L^q(Q_h)} \label{eq:PdualVD} \tag{$P^*_{\sigma}$}\\
& \quad \textrm{s.t.} &&\|\Morrey_{h}(w_{\sigma})\|_{\infty,\bar{\varOmega}_c} \leq \beta \;\text{and} \; \|\Morrey_{\sigma}(w_{\sigma})\|_{\infty,\bar{Q}_c} \leq \alpha \label{eq:inftynorms}
\end{alignat}

\textcolor{C}{Similar to the continuous setting, it can be shown that \eqref{eq:PdualVD} is the Fenchel predual of the problem \eqref{eq:Pvd} restricted to $(u_0,u)\in U_h \times \cU_{\sigma}$.} In order to solve \eqref{eq:PdualVD}, we want to represent $\adjheat_\sigma : \cW_{\sigma} \rightarrow \cY_{\sigma}^* $ by a matrix, as done in \cite{CasasClasonKunisch}. From \cite[Section 4]{DHV} and \cite{Wollner} we know that the matrix representation of $\adjheat_{\sigma}^*: \cY_{\sigma} \rightarrow \cW_{\sigma}^*$ yields a Crank-Nicolson scheme with a smoothing step. We will derive this first.

Let $M_h \coloneqq \nparen{\ninnerprod{ \SpaceNode_{x_j},\SpaceNode_{x_k}}}_{j,k=1}^{N_h}$ be the mass matrix and $\stiff\coloneqq \nparen{ \textstyle \int_{\varOmega} \nabla \SpaceNode_{x_j} \nabla \SpaceNode_{x_k} \, \dd x}_{j,k=1}^{N_h}$ the stiffness matrix corresponding to $Y_h$. 
We define $y_{k,h} \coloneqq y_{\sigma}|_{I_k} \in Y_h$ for $k \in \{1,\ldots,N_{\tau}\} $, $w_{k,h} \coloneqq w_\sigma(\cdot,t_k) \in Y_h$ and $w_{k} \coloneqq w_{k,h} \otimes \TimeNode_{t_k}\in \cW_{\sigma}$ for $k \in \{0,\ldots,N_{\tau}-1\} $. We then obtain the following :
\begin{equation*}
\langle \adjheat_{\sigma}^* y_{\sigma},w_k  \rangle = \nparen{ y_{k+1,h}-y_{k,h} }^{\top} M_h \, w_{k,h} + \nparen{\tfrac{\tau_k}{2} \, y_{k,h} + \tfrac{\tau_{k+1}}{2}\,  y_{k+1,h}}^{\top} \stiff\, w_{k,h} 
\;\;
\text{for all $k \in \{1,\ldots,N_{\tau}-1\}$.}
\end{equation*}
For $k=0$, we have $ \langle \adjheat_{\sigma}^* y_{\sigma},w_0  \rangle = y_{1,h}^{\top} M_h \, w_{0,h} + \tfrac{\tau_1}{2} y_{1,h}^{\top} \stiff \, w_{0,h}$.

In order to represent the discrete state equation \eqref{eq:dse2} by a system of equations, we also need to calculate  $r(w_{\sigma}) \coloneqq \int_{\bar{\varOmega}_c}{w_{\sigma}(0) \, \dd u_{0}} +\int_{\bar{Q}_c}{w_{\sigma} \, \dd u} $. Due to the implicit discrete structure of the controls $(u_0,u)$, \textcolor{Mn}{ we can define $\tilde{u} \in V_h^* \times \cV_{\sigma}^*$ with $[\tilde{u}]_{j,k} = u_{j,k}$ and $[u_{0,h}]_j= u_{j,0}$ for $j \in \cI_h$ and $[u_{\sigma}]_{j,k}= u_{j,k}$ for $(j,k) \in \cI_{\sigma}$. Then we have}
%
%
%
\begin{align*}
	(\Morrey_{h} + \Morrey_{\sigma})^{*}(u_{0,h},u_{\sigma}) 
	= 
	\textstyle
	\sum\limits_{j=1}^{N_h}\sum\limits_{k=0}^{N_{\tau}-1} u_{j,k} \SpaceDelta_{x_j } \otimes \TimeDelta_{t_k} \in \cW_{\sigma}^* 
\end{align*}
with $u_{j,k}=0$ for $(j,0), j \notin \cI_h$ and $(j,k) \notin \cI_{\sigma}$. Since $w_{\sigma}(0)|_{\bar{\varOmega}_c} \in V_h = U_h^*$ and $w_{\sigma}|_{\bar{Q}_c} \in \cV_{\sigma} = \cU_{\sigma}^*$, we get 
\begin{equation*}
	r(w_{\sigma}) 
	= \textstyle \sum\limits_{j \in \cI_h} u_{j,0}w_{j,0} + \sum\limits_{(j,k)\in \cI_{\sigma}}u_{j,k}w_{j,k} 
	= \ninnerprod{ (\Morrey_{h} + \Morrey_{\sigma})^{*}(u_{0,h},u_{\sigma}),w_{\sigma} }_{\cW_{\sigma}^*,\cW_{\sigma}} ,
\end{equation*}

For the remainder of this section, we identify elements from $\mathcal{Y}_{\sigma}$ and $\cW_{\sigma}$ with vectors in $\mathbb{R}^{N_{\sigma}}$, $N_{\sigma} \coloneqq N_h \cdot N_{\tau }$ and elements from $U_h, \cU_{\sigma}$ with vectors in $\mathbb{R}^{|\cI_h|},\mathbb{R}^{|\cI_{\sigma}|}$, respectively. The discrete elements can be expressed via their respective expansion coefficients. To simplify the notation, we define $y_k \coloneqq (y_{1,k},\ldots,y_{N_h,k})^{\top}\in \mathbb{R}^{N_h}$ and write \textcolor{C}{$y_{\sigma}=(y_{1}^{\top},\ldots,y_{N_{\tau}}^{\top})^{\top} \in \mathbb{R}^{N_{\sigma}}$.} Analogously, we define $w_k$ for $k=0,\ldots,N_{\tau}-1$.
%
%
%

\textcolor{C}{We represent the discrete state equation by the following $(N_{\sigma} \times N_{\sigma})$-matrix:}
\begin{equation}
\cL^{\top} \coloneqq
\begin{pmatrix}
\nparen{M_h +\tfrac{\tau_1}{2} \stiff} &0  & \ldots &  \ldots  & 0 
\\
\nparen{-M_h+\tfrac{\tau_1}{2} \stiff} & \nparen{M_h+ \tfrac{\tau_2}{2} \stiff}  & & &\vdots 
\\
0 & \ddots & \ddots
\\
\vdots & & \ddots & \ddots &0 
\\
0 &\ldots & 0 &  \nparen{-M_h+\tfrac{\tau_{N_{\tau}-1}}{2} \stiff} & \nparen{M_h+ \tfrac{\tau_{N_{\tau}}}{2} \stiff}
\end{pmatrix}.
\label{eq:Lsigmastern}
\end{equation}
We point out that $\cL^{\top}$ is in fact the operator $\adjheat_{\sigma}^*$ concatenated with the space-time mass matrix $\cM_{\sigma}$ that maps from $L^p(Q_h)$ to $(L^q(Q_h))^*$, which is an important detail for the implementation. A~representation of $\adjheat_{\sigma}$, also concatenated with a space-time mass matrix, is the matrix $\cL = (\cL^{\top})^{\top}$. If we now actually want the representative of $\adjheat_{\sigma} w_{\sigma}$, it is necessary to multiply with the inverse of the space-time mass matrix $\cM_{\sigma}^{-1}$. 

Furthermore the embedding $(\Morrey_{h} \oplus \Morrey_{\sigma}): \cW_{\sigma} \rightarrow (V_h \times \cV_{\sigma})$, defined in \eqref{eq:Morreyh} and \eqref{eq:Morreysigma}, can be represented by a restriction matrix:
\begin{equation*}
(\cR_{h} +  \cR_{\sigma}): \mathbb{R}^{N_{\sigma}} \rightarrow \mathbb{R}^{|\cI_h| + |\cI_{\sigma}|},  \qquad w_{\sigma} \mapsto ( (w_{0,j})_{j\in \cI_h}^{\top}, (w_{j,k})_{(j,k)\in \cI_{\sigma}}^{\top})^{\top}.
\end{equation*}
From duality we conclude that $(\Morrey_{h} \oplus \Morrey_{\sigma})^*$ can be represented by $(\cR_{h} + \cR_{\sigma})^{\top}$, such that $\cL^{\top} y_{\sigma} = (\cR_{h} + \cR_{\sigma})^{\top}(u_{0,h}^{\top},u_{\sigma}^{\top})^{\top}$ is the matrix vector formulation of \eqref{eq:dse2}. 

We equivalently reformulate the constraints \eqref{eq:inftynorms} in \textcolor{C}{\ref{eq:PdualVD}} using \eqref{eq:betanorm} and \eqref{eq:alphanorm}:
\begin{alignat}{3}
\max_{j \in \cI_h} |w_{j,0}| & \leq \beta \quad &&\text{and} \qquad\qquad\quad\; \max_{(j,k) \in \cI_{\sigma}}|w_{j,k}| &\leq \alpha, \notag\\
	\text{if and only if}
	\quad 
	\max_{j \in \cI_h} \{w_{j,0}, -w_{j,0}\}- \beta &\leq 0  \quad &&\text{and} \quad \max_{(j,k) \in \cI_{\sigma}}\{w_{j,k},-w_{j,k}\}-\alpha \, &\leq 0. \notag
\end{alignat}
We can now formulate linear inequality constraints that are equivalent to \eqref{eq:inftynorms}.
All inequalities are strictly fulfilled for $w_\sigma = 0$, thus $w_\sigma = 0$ is an interior point of the feasible set and thus the Slater condition is satisfied (see e.g. \cite[(1.132)]{HPUU}). 
%
%
We discretize the desired state by sampling it on the dual time grid:
\begin{align}
	y_{\operatorname{d},\sigma} 
	=  
	\textstyle \sum\limits_{j=1}^{N_h}\sum\limits_{k=1}^{N_{\tau}} 
	\DesiredState \bigparen{x_j,(t_{k-1}+t_{k})/2} \, \SpaceNode_{x_j } \otimes \chi_k \in \cY_{\sigma}.
	\label{eq:DiscreteDesiredState}
\end{align}
By doing so, we assume that $\DesiredState$ has a certain minimum continuity. However, discretization by local averaging is also possible.
If we make sure that $y_{\operatorname{d},\sigma}  \to \DesiredState$ for $\nabs{\sigma} \to 0$, then using $y_{\operatorname{d},\sigma}$ in \eqref{eq:Pvd} instead of $\DesiredState$ does not interfere with the convergence result Theorem~\ref{thm:DiscVarConvergence}.

We proceed by setup the corresponding Lagrangian $\mathscr{L}$ with multipliers \textcolor{M2}{$\lambda^{(1)},\lambda^{(2)} \in \mathbb{R}^{|\cI_{\sigma}|}$ and $\lambda^{(3)},\lambda^{(4)} \in \mathbb{R}^{|\cI_h|}$: 
\begin{alignat}{2}
\mathscr{L}(w_{\sigma},\lambda^{(1)}, \lambda^{(2)},\lambda^{(3)},\lambda^{(4)}) 
&= \textstyle \tfrac{1}{p} \|\cM_{\sigma}^{-1} \cL w_{\sigma}\|_{L^p(Q_h)}^p + \langle &&\cM_{\sigma}^{-1}\cL w_{\sigma}, y_{\operatorname{d},\sigma}\rangle_{L^p(Q_h),L^q(Q_h)} \notag \\
& \quad \textstyle + \sum\limits_{(j,k)\in \cI_{\sigma}} \lambda_{j,k}^{(1)} \nparen{w_{j,k} -\alpha } 
&&+ \textstyle\sum\limits_{(j,k)\in \cI_{\sigma}} \lambda_{j,k}^{(2)} \nparen{ -w_{j,k} -\alpha } \notag \\
&\quad 
\textstyle
+ \sum\limits_{j \in \cI_h} \lambda_{j}^{(3)} \nparen{ w_{j,0} -\beta } 
&&+\textstyle \sum\limits_{j \in \cI_h} \lambda_{j}^{(4)} \nparen{-w_{j,0} -\beta } \notag.
\end{alignat}}
\textcolor{C}{For the sake of simplified numerics, we use a lumped mass matrix approach for computing $K_{\sigma}(w_{\sigma})$, where $w_{\sigma}$ is a vector, i.e., with a slight abuse of notation, we employ
\begin{align*}
\|\cM_{\sigma}^{-1}\cL w_{\sigma}\|_{L^p(Q_h)}^p &\coloneqq \textstyle \sum\limits_{j=1}^{N_h} \sum\limits_{k=1}^{N_{\tau}} |(\cM_{\sigma}^{-1}\cL w_{\sigma})_{j,k}|^p\, \omega_j\, \tau_k,  \\
\langle \cM_{\sigma}^{-1}\cL w_{\sigma}, y_{\operatorname{d},\sigma}\rangle_{L^p(Q_h),L^q(Q_h)} &\coloneqq \textstyle \sum\limits_{j=1}^{N_h} \sum\limits_{k=1}^{N_{\tau}} (\cM_{\sigma}^{-1}\cL w_{\sigma})_{j,k} \DesiredState(x_j,t_k) \, \omega_j\, \tau_k,
\end{align*} 
where $\omega_j \coloneqq \int_{\varOmega} \SpaceNode_{x_j } \dd x$.}
We also use a lumped mass matrix approach for $\cM_{\sigma}^{-1}$.

%
%
We can now form the optimality system using the Karush-Kuhn-Tucker conditions (see, e.g., \cite[(5.49)]{Boyd}).
Since the Slater condition is satisfied,
the Karush-Kuhn-Tucker conditions state that at the minimum $w_\sigma$, 
there must be $\lambda^{(1)}, \lambda^{(2)},\lambda^{(3)},\lambda^{(4)}$ the partial differential $\tfrac{\partial \mathscr{L}}{\partial w_{\sigma}}$ of $\mathscr{L}$ at the point $(w_\sigma,\lambda^{(1)}, \lambda^{(2)},\lambda^{(3)},\lambda^{(4)})$ has to vanish 
and that the following complementary conditions have to be fulfilled:
\textcolor{M2}{\begin{alignat}{6}
	\lambda_{j,k}^{(i)} 
	\Bigparen{ \tfrac{w_{j,k}}{(-1)^{(i-1)}}  -\alpha } &=0 
	\;\; \text{and} \;\; 
	\lambda_{j,k}^{(i)} &&\geq 0 
	\;\; &&\text{and} \;\; 
	\Bigparen{ \tfrac{w_{j,k}}{(-1)^{(i-1)}}  -\alpha }  &&\leq \; && 0 \quad &&\forall (j,k) \in \cI_{\sigma} ,i \in \set{1,2},  \notag 
	\\
	\lambda_{j}^{(i)} 
	\Bigparen{ \tfrac{w_{j,0}}{(-1)^{(i-1)}}  -\beta } &=0 
	\; \; \text{and} \;\; 
	\lambda_{j}^{(i)} &&\geq 0 
	\;\; &&\text{and} \;\;  
	\Bigparen{ \tfrac{w_{j,0}}{(-1)^{(i-1)}} -\beta } &&\leq \; && 0 \quad &&\forall j \in \cI_h, i \in \set{ 3,4 },  \notag 
	\end{alignat}}
which can be equivalently reformulated for all $(j,k)\in \cI_{\sigma}$ and $j \in \cI_h$ \textcolor{M1}{with an arbitrary $\kappa >0$}:
\textcolor{M2}{\begin{alignat}{4}
	&N^{(1)}_{j,k} &&\coloneqq   \max \Set{0, \, \lambda^{(1)}_{j,k} + \kappa (w_{j,k}-\alpha) } &&- \lambda^{(1)}_{j,k} &&=0, \notag \\
	& N^{(2)}_{j,k} &&\coloneqq   \max\Set{0, \, \lambda^{(2)}_{j,k} + \kappa (-w_{j,k}-\alpha) }&&- \lambda^{(2)}_{j,k}&&=0,  \notag 
	\\
	&N^{(3)}_{j}   &&\coloneqq  \max\Set{0, \, \lambda^{(3)}_{j} + \kappa (w_{j,0}-\beta ) } &&-\lambda^{(3)}_{j} &&=0, \notag \\ 
	& N^{(4)}_{j}   &&\coloneqq   \max\Set{0, \, \lambda^{(4)}_{j} + \kappa (-w_{j,0}-\beta) }&&-\lambda^{(4)}_{j} &&=0. \notag
	\end{alignat}}
We define \textcolor{M2}{
\begin{equation*}
	\mathscr{F}(w_{\sigma},\lambda^{(1)},\lambda^{(2)},\lambda^{(3)},\lambda^{(4)}) 
	\coloneqq 
	\begin{pmatrix} \tfrac{\partial \mathscr{L}}{\partial w_{\sigma}} & N^{(1)} & N^{(2)} & N^{(3)} & N^{(4)}
	\end{pmatrix}^{\top} \in \mathbb{R} ^{N_{\sigma}  + 2 (|\cI_{\sigma}| + |\cI_h|)}  
\end{equation*}} 
%
containing the left sides of our optimality system and solve the equation \linebreak \textcolor{M2}{${\mathscr{F}(w_{\sigma},\lambda^{(1)},\lambda^{(2)},\lambda^{(3)},\lambda^{(4)})=0}$} by a semismooth Newton method (\cite[Algorithm 2.11]{HPUU}). In this algorithm we need to choose the matrix to be used in the semismooth Newton equation from the set of matrices denoted by Clarke's generalized Jacobian (\cite[Example 2.4]{HPUU}):
\begin{equation*}
\partial \mathscr{F} (x) \coloneqq \operatorname{conv}
	\Big\{
	 M \;:\; x_k \stackrel{k \rightarrow \infty}{\longrightarrow} x, 
	 \, 
	 \mathscr{F}'(x_k) \longrightarrow M,\, 
		\mathscr{F} \text{ differentiable at } x_k 
	 \Big\},
\end{equation*}
where $\conv$ denotes the convex hull.
Here, a choice has to be made for the numerics since the generalized Jacobians of
$N^{(i)}$, $i \in \set{1,2,3,4 }$ need not be singletons.
This is due to the $\max$-functions and because we have
\begin{equation*}
\partial_x \bigparen{ \max\set{0, g(x) } } = 
\begin{cases} 0, & \text{if } g(x)<0, \\ 
\textcolor{M1}{\operatorname{conv}\Set{ 0, \partial_x g(x) },} & \text{if }  g(x) =0,  \\ 		
\partial_x g(x), & \text{if } g(x)>0
\end{cases}
\end{equation*}
for every differentiable scalar function $g(x)$.
Here, we make the decision to always choose \textcolor{Cn}{$\partial_x \bigparen{ \max\set{0, g(x) } } = \partial_x g(x)$, if $g(x)=0$.} 
%
Using this, we define:
\textcolor{M2}{\begin{equation}
 D \mathscr{F} \coloneqq \begin{pmatrix} &\tfrac{\partial^2 \mathscr{L}}{\partial^2 w_{\sigma}^2} &\tfrac{\partial^2 \mathscr{L}}{\partial w_{\sigma} \partial \lambda^{(1)}} &\tfrac{\partial^2 \mathscr{L}}{\partial w_{\sigma} \partial \lambda^{(2)}} &\tfrac{\partial^2 \mathscr{L}}{\partial w_{\sigma} \partial \lambda^{(3)}}  &\tfrac{\partial^2 \mathscr{L}}{\partial w_{\sigma} \partial \lambda^{(4)}}\\
	&\tfrac{\partial N^{(1)}}{\partial w_{\sigma}} &\tfrac{\partial N^{(1)}}{\partial \lambda^{(1)}}  &0 &0 &0 \\
	&\tfrac{\partial N^{(2)}}{\partial w_{\sigma}} &0  &\tfrac{\partial N^{(2)}}{\partial \lambda^{(2)}} &0 &0 \\
	&\tfrac{\partial N^{(3)}}{\partial w_{\sigma}} &0  &0 &\tfrac{\partial N^{(3)}}{\partial \lambda^{(3)}} &0 \\
	&\tfrac{\partial N^{(4)}}{\partial w_{\sigma}} &0  &0 &0 &\tfrac{\partial N^{(4)}}{\partial \lambda^{(4)}} 
	\end{pmatrix} \in \partial \mathscr{F} ,\label{eq:matrixDF}
	\end{equation}}%
for the semismooth Newton method. An interesting observation is that for $\kappa=1$ we have a symmetric matrix on the active sets.

We want to remark that we could follow \cite{CasasClasonKunisch} and employ Fenchel duality to recover a problem in the variable $\tilde{u}$. However, this would require to add the representation of the adjoint from \eqref{eq:upartialF} to our optimality system. As $p>2$ the exponent $\tfrac{1}{p-1}$ is strictly smaller than 1, which is problematic for derivative based methods and was our main motivation to use the Fenchel duality approach in the first place. 

Instead, we solve for the optimal adjoint $\bar{w}_{\sigma}$ and recover the optimal control $(\bar{u}_{0,h},\bar{u}_{\sigma})$ through the discrete version of \eqref{eq:RecoverControl}:
\begin{equation}
(\cR_{h} + \cR_{\sigma})^{\top} (\bar{u}_{0,h}^{\top},\bar{u}_{\sigma}^{\top})^{\top} = \cL^{\top} (| \cM_{\sigma}^{-1}\cL \bar{w}_{\sigma}|^{p-2} \, 
\cM_{\sigma}^{-1} \cL \bar{w}_{\sigma} + y_{\operatorname{d},\sigma})
\label{eq:uweinsdis}
\end{equation}
One could come to the conclusion that this is problematic since $(\cR_{h} + \cR_{\sigma})^{\top}$ is in general only injective, not surjective.
But the expansion coefficients $\bar u_{j,0}$ and $\bar u_{j,k}$ of the discrete solution
$(\bar{u}_{0,h},\bar{u}_{\sigma})$ have to vanish anyways for
$j \not \in \cI_h$ and $(j,k) \not \in \cI_\sigma$ (see \eqref{eq:DiscreteSparsityStructure}). 
So the remaining coefficients have merely to be read off.

\section{Discontinuous Galerkin Discretization}
\label{sec:discgalerkin}
This section deals with a full discretization concept of \eqref{eq:P}, namely \textcolor{C}{discontinuous Galerkin discretization}, which is suggested in \cite{CasKun}. 
The discretization strategy will be adapted to our setting and notation. 
Also a convergence result for the fully discrete problem \eqref{eq:Psigma} analogous to Theorem \ref{thm:DiscVarConvergence} is proven in \cite{CasKun}.


We use the discrete spaces $\cY_{\sigma}, Y_h, U_h$ as introduced before in \eqref{eq:Ysigma}, \eqref{eq:Yh}, \eqref{eq:Uh}, respectively and the space-time discrete control space:
\begin{equation*}
	\cU_{\operatorname{DG}} 
	\coloneqq 
	\Span  \, \{\SpaceDelta_{x_j } \otimes \chi_k :j \in \cI_h, k \in \cI_{\tau} \}, \qquad \cI_{\tau} \coloneqq \{ k : I_k \subset \bar{I}_c\} 
.
\end{equation*}

In \cite{CasKun}, an \textcolor{C}{implicit Euler time stepping scheme} is used for the discrete state equation. 
For $y_{\sigma} \in \cY_{\sigma}$ and for every $k \in \set{1,\ldots,N_{\tau}}$, we define $y_{k,h} \coloneqq y_{\sigma}|_{I_k} \in Y_h$. 
Let $(u_{0,h},u_{\sigma}) \in  U_h\times\mathcal{U}_{\operatorname{DG}} $ be given and $z_h \in Y_h$ arbitrary. 
Then the following equations form the discrete state equation:
\begin{equation}
	\begin{cases}
		\textcolor{M1}{  \left\langle y_{k,h}-y_{k-1,h} , z_h \right\rangle_{L^2}  } + \tau_k \int_{\varOmega}{\nabla y_{k,h} \nabla z_h \, dx} 
		=  
		\int_{\bar Q_c}{ (z_h \otimes \chi_k) \,\dd u_{\sigma}}
		\quad
		\text{for $k \in \set{1,\ldots,N_{\tau}}$,}
		\\
		y_{0,h}=y_{0h},
	\end{cases}
	\label{eq:dse}
\end{equation}
where $y_{0h} \in Y_h$ is the unique element satisfying:
\begin{equation}
	\textstyle
	\ninnerprod{ y_{0h},z_h }_{L^2} = \int_{\bar{\varOmega}_c}{z_h \, \dd u_{0,h}} \qquad \forall \, z_h \in Y_h \label{eq:y0h}\\
\end{equation}
Here $\left\langle \cdot,\cdot\right\rangle_{L^2}\ $ denotes the scalar product in $L^2(\varOmega)$. 
We denote the solution of the discrete state equation \eqref{eq:dse}
by  $y_{\sigma}(u_{0h},u_{\sigma})$,
and define the discrete objective function
\begin{align*}
	J_{\operatorname{DG}}(u_{0h},u_{\sigma})
	\ceq 
	\tfrac{1}{q} \, \nnorm{y_{\sigma}(u_{0h},u_{\sigma})-\DesiredState }_{L^q(Q_h)}^{q} 
	+ \alpha \, \nnorm{u_{\sigma} }_{\mathcal{M}(\bar{Q}_c)}
	+ \beta \, \nnorm{ u_{0h} }_{\mathcal{M}(\bar{\varOmega}_c)}.
\end{align*}
This allows us to formulate the following discrete optimization problem
\begin{equation}
	\min_{(u_{0h},u_{\sigma}) \in U_h \times \mathcal{U}_{\operatorname{DG}} } J_{\operatorname{DG}}(u_{0h},u_{\sigma}) 
	\tag{$P_{\operatorname{DG}}$}.
	\label{eq:Psigma}
\end{equation}
Similar to \cite{CasasClasonKunisch}, we set up the system matrix for the discrete state equation. 
One difference that we need to consider is $u_{0,h} \neq 0$.
This leads to $N_h$ further degrees of freedom for the state
and also to $N_h$ additional columns and $N_h$ additional rows in the system matrix; see \cite[Chapter 12]{Thomee} for further details.
With the mass matrix $M_h \coloneqq \nparen{\ninnerprod{ \SpaceNode_{x_j},\SpaceNode_{x_k}}}_{j,k=1}^{N_h}$ and
the stiffness matrix $\stiff\coloneqq \nparen{ \textstyle \int_{\varOmega} \nabla \SpaceNode_{x_j} \nabla \SpaceNode_{x_k} \, \dd x}_{j,k=1}^{N_h}$,
the left hand sides of \eqref{eq:dse} and \eqref{eq:y0h}
can be encoded into the following matrix of size $(N_\sigma + N_h) \times (N_\sigma + N_h)$:
\begin{align*}
\cL_{\operatorname{DG}}^{\top}
\ceq
\begin{pmatrix} &M_h &0 &\ldots&\ldots  &0\\
&-M_h &M_h + \tau_1 \stiff  &  & &\vdots \\
&0 &-M_h & M_h +\tau_2 \stiff &  &\vdots\\
&\vdots & &\ddots &\ddots &0
\\
&0 &\ldots &0 &-M_h &M_h +\tau_{N_{\tau}}\stiff  
\end{pmatrix}.
\end{align*}
As in Section \ref{sec:vardiscr}, this matrix represents the state-to-control-operator concatenated with the space-time mass matrix $\cM_{\sigma}$. The representation of the adjoint state equation is $\cL_{\operatorname{DG}} = (\cL_{\operatorname{DG}}^{\top})^{\top}$. 
 
With the discrete representations
\begin{align*}
 	\textstyle
	u_{\sigma} 
	= \sum\limits_{j \in \cI_h} \sum\limits_{i \in \cI_{\tau}} u_{j,i} \,\SpaceDelta_{x_j} \otimes \chi_i \in \cU_{\operatorname{DG}}
	\qand
	z_h  
	= \sum\limits_{l=1}^{N_h} z_l \, \SpaceNode_{x_l} \in Y_h
\end{align*}
and using that the ``mass matrix'' $\nparen{\ninnerprod{ \SpaceDelta_{x_j},\SpaceNode_{x_l} }}_{j,l=1}^{N_h}$ is the identity in $\mathbb{R}^{N_h \times N_h}$, we obtain the following for the right hand side in \eqref{eq:dse}:
\begin{equation*}
	\textstyle
	\int_{\bar{Q}_c} (z_h \otimes \chi_k) \,\dd u_{\sigma} 	
	=  
	\begin{cases}
	\tau_k \sum\limits_{j \in \cI_h} u_{j,k} \, z_j, &\text{if $k \in \cI_{\tau}$,}  \\
	0, &\text{else}.
	\end{cases}
\end{equation*}
Analogously, with $u_{0,h} = \textstyle\sum\limits_{j \in \cI_h} u_j \, \SpaceDelta_{x_j} \in U_h$, the right hand side from \eqref{eq:y0h} turns into
\begin{equation*}
	\textstyle\int_{\bar{\varOmega }_c} z_h \, \dd u_{0,h} = \sum\limits_{j \in \cI_h} u_j \, z_j.
\end{equation*}
Restriction matrices similar the those in \ref{sec:vardiscr} can be derived and the discrete state equation can be written in matrix form. 
This can be used to discretize \eqref{eq:Pdual}, leading to an optimization problem in the variable $w_{\operatorname{DG}} = (w_{j,k})_{j=1, k=0}^{N_h, N_{\tau}} \in \mathbb{R}^{N_{\sigma}+N_h}$. 
The setup for the fully discrete problem and the derivation of the optimality system are almost identical to the procedures from Section \ref{sec:vardiscr};
one only has to replace $\cL$ by $\cL_{\operatorname{DG}}$
and to keep in mind that the number of degrees of freedom changes from $N_{\sigma}$ to $N_{\sigma}+N_h$.

\section{Computational Results}\label{sec:compres}


\newlength{\imgwidth}
\setlength{\imgwidth}{0.315\textwidth}
	
	We numerically solve \eqref{eq:PdualVD} by a semismooth Newton method as derived in Section \ref{sec:vardiscr}. To simplify, we fix $u_0 =0$. Therefore the condition $\|\Morrey_{h}(w_{\sigma}) \|_{\infty,\bar{\varOmega}_c} \leq \beta$ in problem \eqref{eq:PdualVD} disappears and so do $\lambda^{(3)}$ and $\lambda^{(4)}$ in the Lagrangian $\mathscr{L}$. The dimension of the optimality system is reduced accordingly since $N^{(3)}$ and $N^{(4)}$ do not have to be considered. For the discontinuous Galerkin discretization from Section \ref{sec:discgalerkin} of problem \eqref{eq:Pdual}, we proceed similarly. Furthermore, here the first row and column of $\cL_{\operatorname{DG}}^{\top}$ can be eliminated.
	
	In this section all variables are specified as their discrete representatives, hence we omit the indices. As our domain for both examples, we choose 
$\varOmega = {(0,1)} \subset \mathbb{R}$, $I={(0,\tfrac{3}{2})}$, and the relatively compact Lipschitz domain $Q_c \coloneqq (\tfrac{1}{4},\tfrac{3}{4}) \times (\tfrac{1}{4},\tfrac{5}{4}) \subset \subset Q = \varOmega \times I $. 
We assume that our mesh is equidistant, consequently every cell is of size $\tau \cdot h$. We set $\kappa = 1$ and $q=\tfrac{4}{3}$ so that $p=4$. 

\textcolor{C}{Let us remark that $p >2$ can lead the the matrix $D\mathscr{F}(w_\sigma,\lambda)$ being singular. 
The cause of this trouble is the second derivative of $z \mapsto \tfrac{1}{p} \, \nnorm{z}_{L^p(Q)}^p$; it appears as central building block of $\frac{\partial^2 \mathscr{L}}{\partial^2 w^2}$ and is nearly singular whenever $z$ is not pointwise bounded away from $0$.
We circumvent this problem by adding a suitable multiple of the residual $\nnorm{\mathscr{F}} \, \cM_\sigma$ to the second derivative of $z \mapsto \tfrac{1}{p} \, \nnorm{z}_{L^p(Q)}^p$; as it is well-known (see, e.g., \cite{MR3800474}), such a regularization does not deteriorate the convergence rate of Newton's method.
}

The purpose of our first numerical example is to illustrate the differences between variational discretization and discontinuous Galerkin discretization.
Therefore, we use a relatively coarse space-time grid with $h = \frac{1}{4}$ and $\tau = \tfrac{h}{2}= \frac{1}{8}$. \textcolor{M2}{We generate a discrete desired state $y_{\operatorname{d}}$ by setting 
$\DesiredState \ceq y(u) = L^{-*} \Morrey^*(u)$ for the measure control $u = \updelta_{(1/2,1/2)}$;
afterwards, we discretize $\DesiredState$ according to \eqref{eq:DiscreteDesiredState}, where we utilize a truncated Fourier expansion in order to evaluate $\DesiredState$ on the points $(x_j , (t_{k-1}+t_k)/2)$.
}
Consequently, this  problem is a source identification example that inherits sparsity. 
If the penalty parameter $\alpha$ equals zero, the only admissible point for the predual problem is $w \equiv 0$. Hence, \eqref{eq:uweinsdis} shows that in this case the optimal discrete controls $u_{\sigma,0} $ and $u_{\operatorname{DG},0}$ for the two discretization approaches can be calculated by applying the discrete heat operator to $\DesiredState$.
This is meaningful since, for $\alpha =0$, the only term remaining in the objective functional is the tracking term $\tfrac{1}{q}\|y-\DesiredState\|_{L^q(Q)}^q$. \textcolor{C}{In this sense the controls $u_{\sigma,0} $ and $u_{\operatorname{DG},0}$ are the solutions to \eqref{eq:Pvd} and \eqref{eq:Psigma} for $\alpha = 0$ with the chosen $\DesiredState$.} Due to the different discretization approaches the calculated controls differ, which can be observed in Figure \ref{fig:NumericalSetup}. \textcolor{M2}{As a consequence of the discretization error in $\DesiredState$, we are not able to reproduce $u$ exactly in either case.}

\begin{figure}[t]
	\begin{center}
		\setlength{\tabcolsep}{0pt}
		\begin{tabular}{|c|c|c|c|c|}
			\hline
			\textbf{control}
			& \textbf{associated state}
			& \textbf{desired state}
			& \textbf{\eqref{eq:Pvd}, $\alpha = 0$}
			& \textbf{\eqref{eq:Psigma}, $\alpha = 0$}
			\\
			\hline
			\raisebox{-0.5\height}{
			\includegraphics[width=0.185\textwidth]{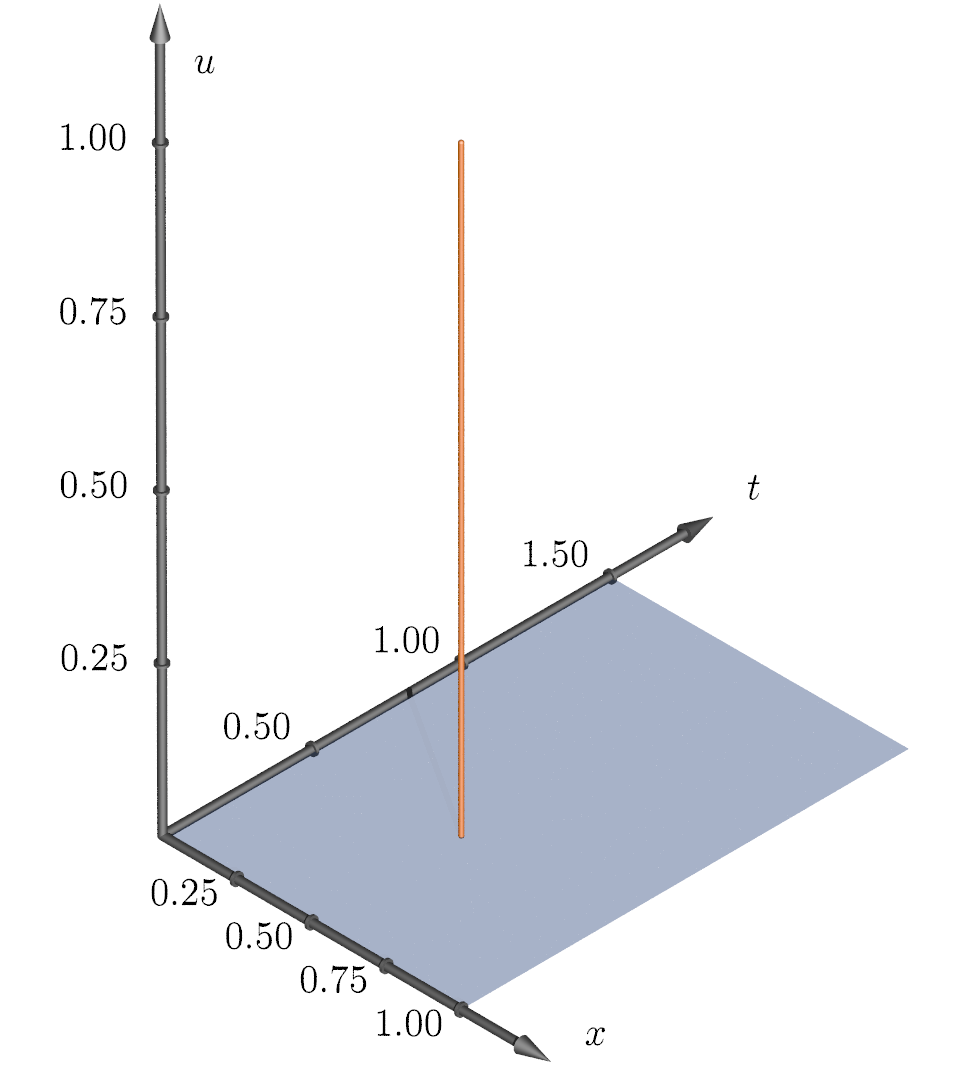}
			}
	     	&\raisebox{-0.5\height}{	
			\includegraphics[width=0.185\textwidth]{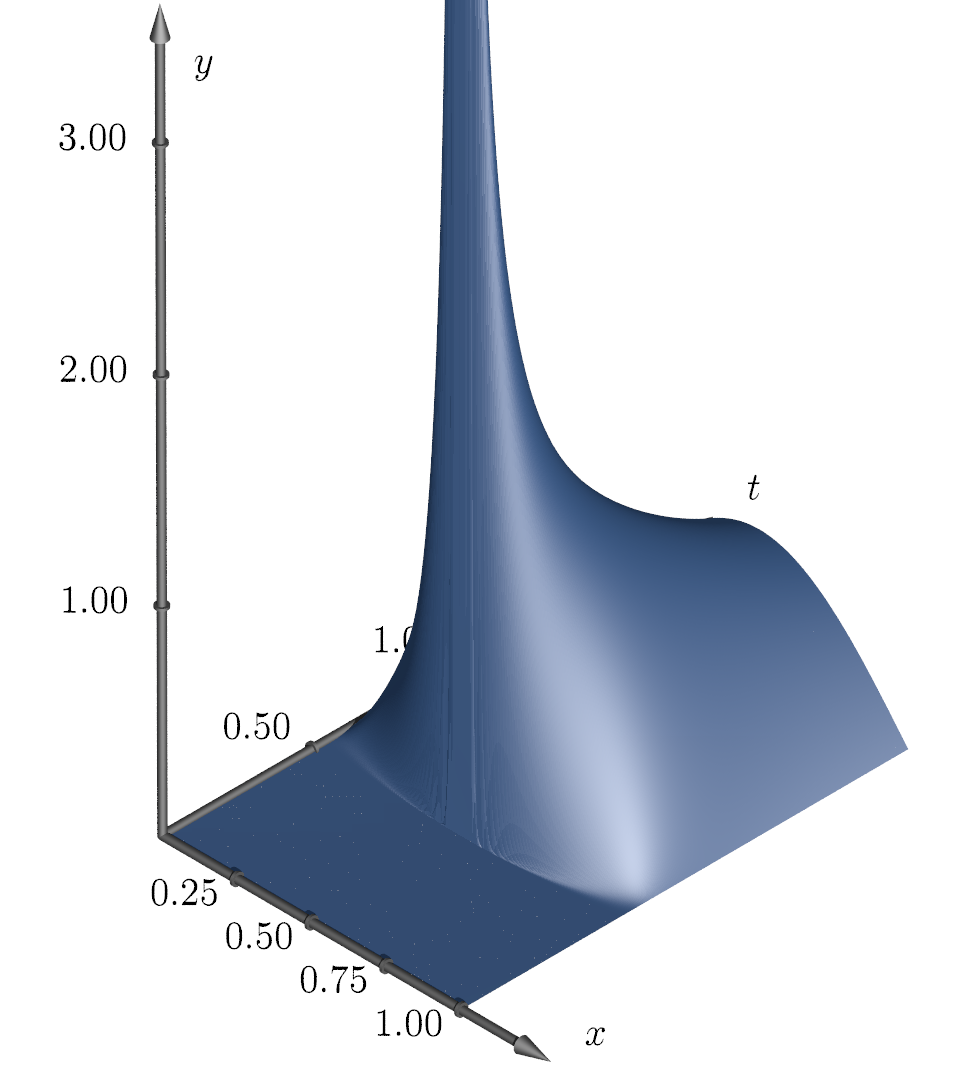}
		    }
	        &\raisebox{-0.5\height}{	
		   	\includegraphics[width=0.185\textwidth]{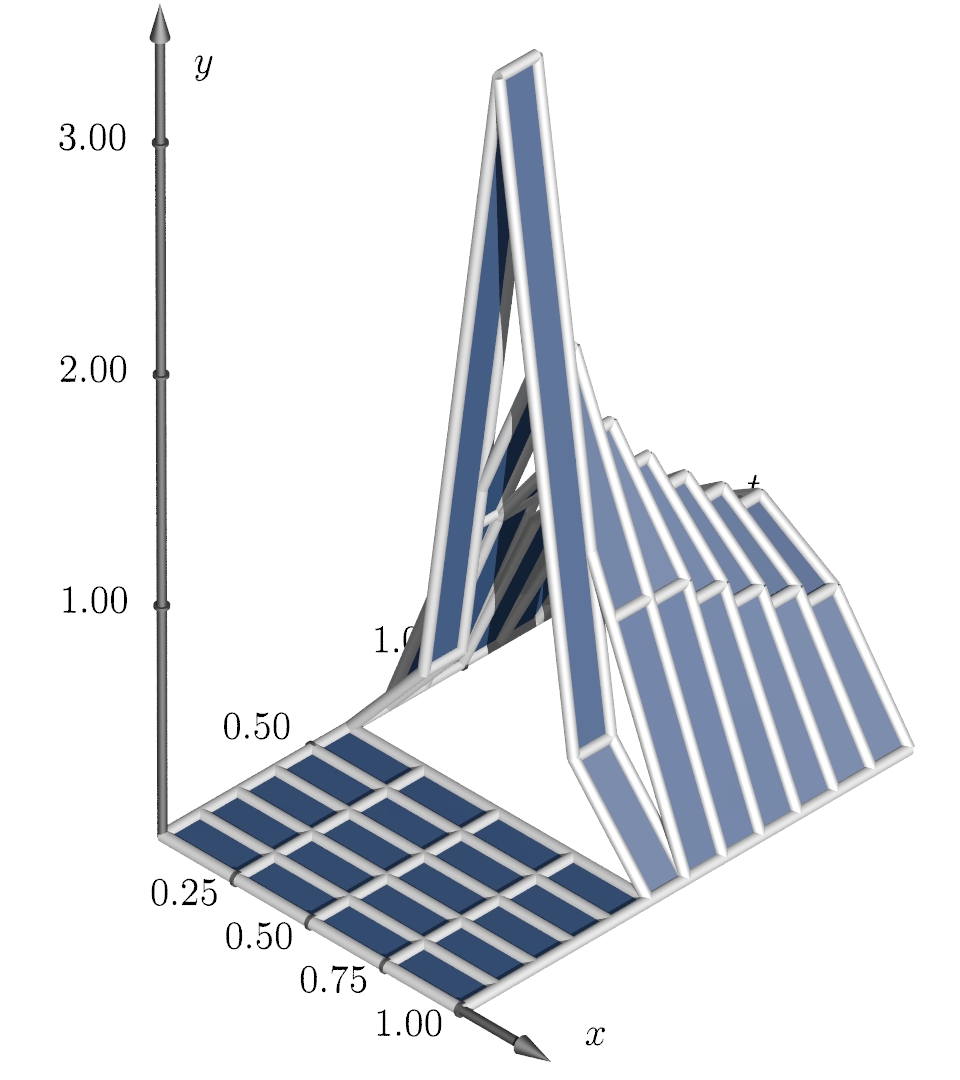}
			}
			&\raisebox{-0.5\height}{	
			\includegraphics[width=0.185\textwidth]{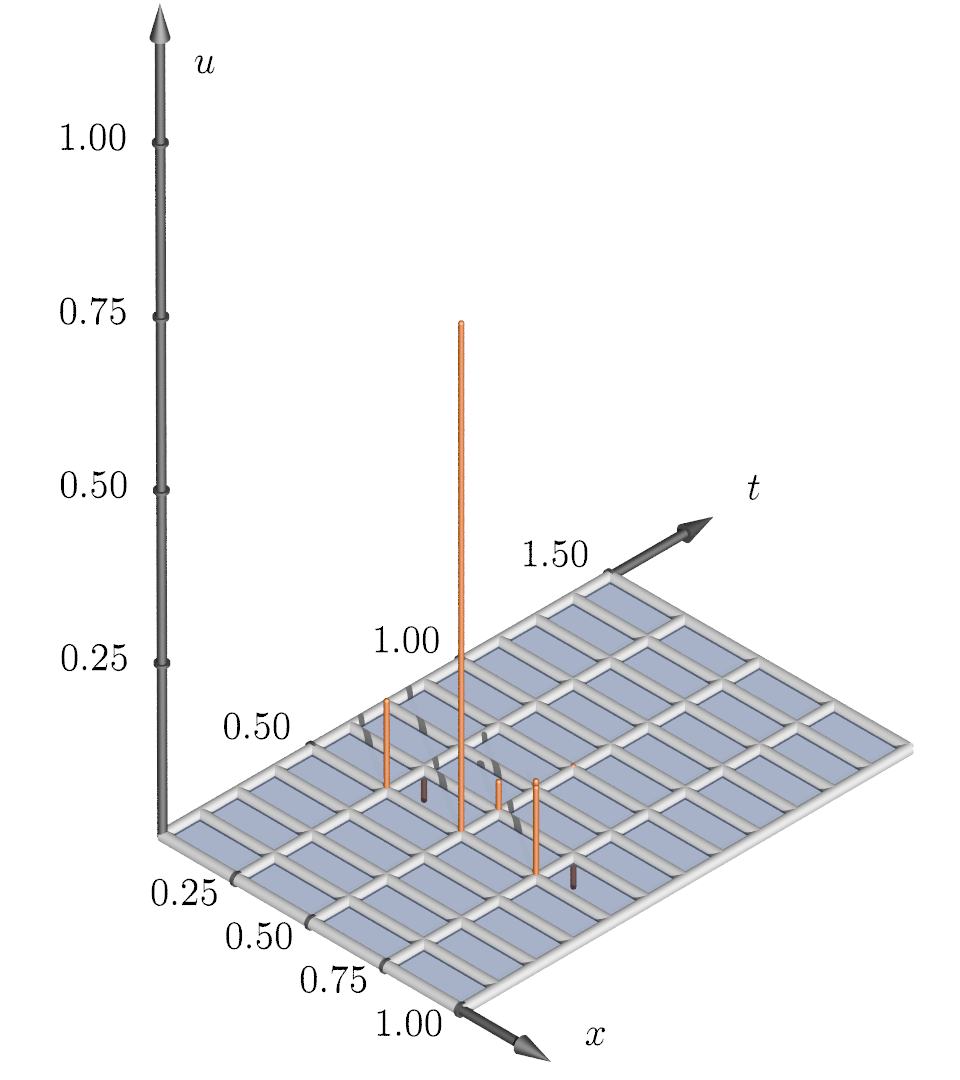}
			}
			&\raisebox{-0.5\height}{
			\includegraphics[width=0.185\textwidth]{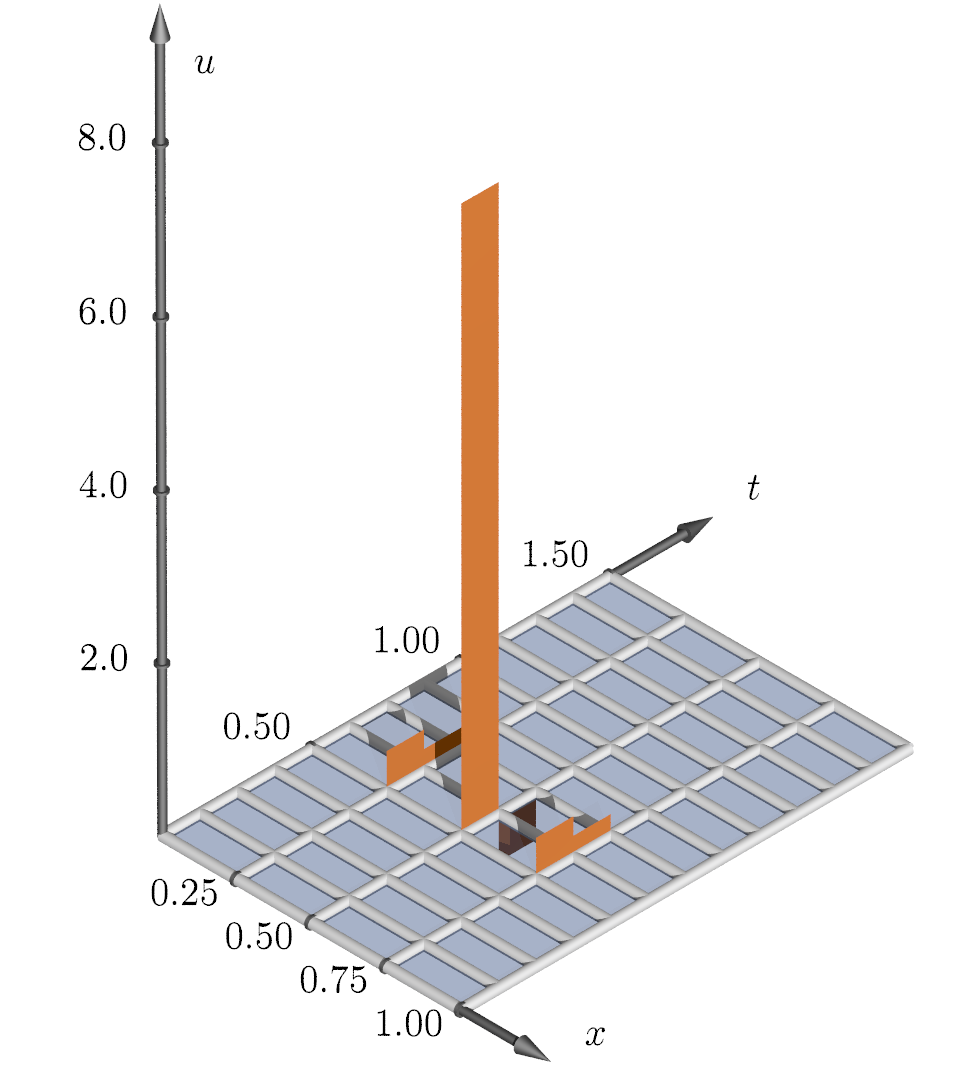}
			}	
			\\
			\hline	
		\end{tabular}
	\end{center}
	\caption{Numerical setup on $4 \times 12$ space-time grid with $q = \tfrac{4}{3}$. From left to right: control $u = \updelta_{(1/2,1/2)}$, associated state $y(u)$ (sampled  from the analytic solution with spacial Fourier modes), discrete desired state $\DesiredState$ and calculated controls $u_{\sigma,0}$ and $u_{\operatorname{DG},0}$ for $\alpha =0$. \\
		Here the controls are represented by their coefficients. 
		In the case of piecewise constant controls in time, which are used in the discontinuous Galerkin setting, this leads to coefficient values, which are scaled with $\tfrac{1}{\tau}=8$. 
	}
	\label{fig:NumericalSetup}
\end{figure}
	To identify the source location, we raise the penalty parameter $\alpha$ because this will lead to a decrease in the norm of the control and we expect a smaller support. The influence of $\alpha$ can be observed by plotting the norm of $u_{\sigma, \alpha}$ and $u_{\operatorname{DG},\alpha}$ respectively for a range of $\alpha$. 
For each \textcolor{Mn}{$i \in \{ \sigma, \operatorname{DG} \}$}, there exists a value $\bar{\alpha}_i$, such that for all $\alpha_i \geq \bar{\alpha}_i$ the optimal control corresponding to $y_{\operatorname{d},\sigma}$ is $u_{i,\alpha_i} \equiv 0$. Additionally it is interesting to look at the values of 
$\nnorm{y_{i,\alpha}-y_{\operatorname{d}}}_{L^{4/3}}$ for $i \in \{\sigma,\operatorname{DG}\}$ and various values of $\alpha$; we plotted the dependences in Figure \ref{fig:uNorm_and_yerror}.
	According to our expectations, the control norms are monotonically decreasing in $\alpha$ and eventually go to zero, while the errors in the tracking terms $\nnorm{y_{i,\alpha}-y_{\operatorname{d}}}_{L^{4/3}}$ grow. The graphs for both strategies look very similar. This makes perfect sense, as we discretize the same problem and both discretization strategies converge towards the true solution.

	%
	\begin{figure}[ht]
		\begin{center}
			\begin{minipage}[c]{0.49\textwidth}
				\includegraphics[width=\textwidth]{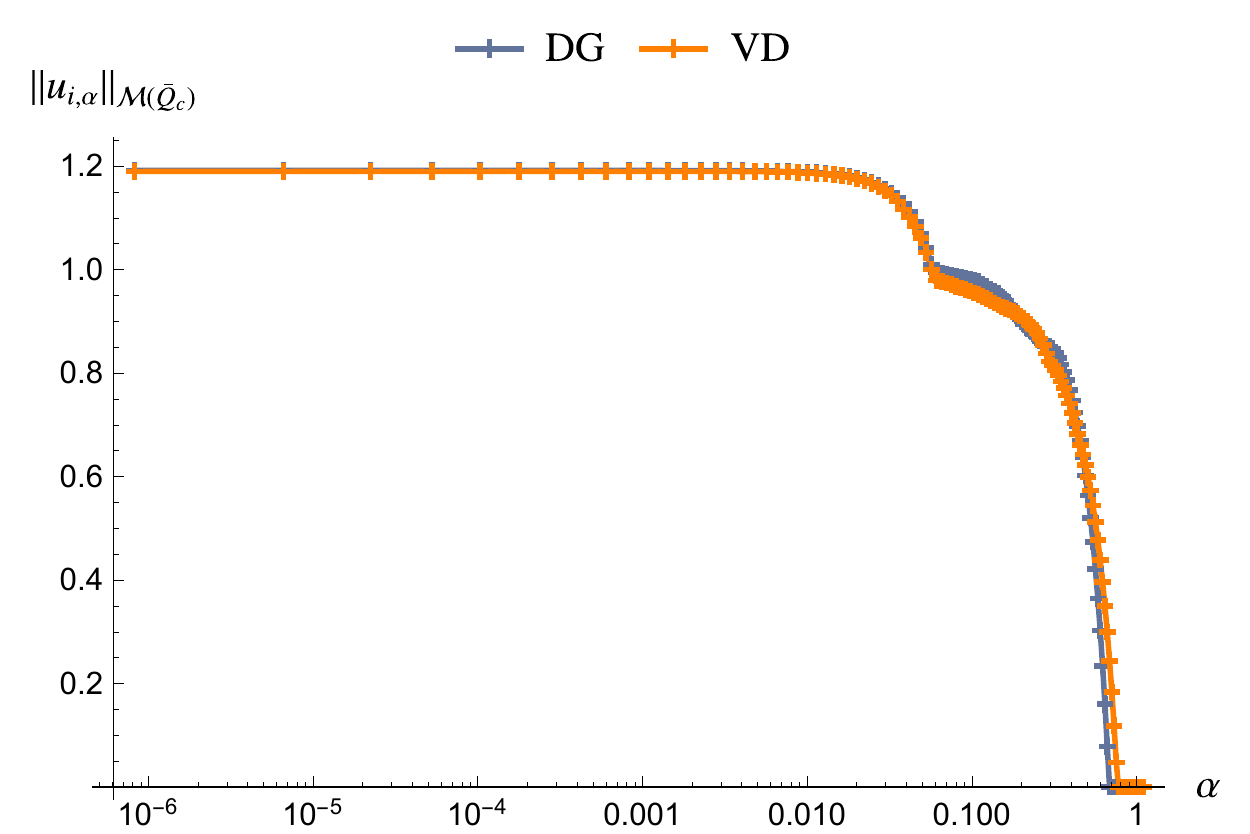}
			\end{minipage}
			\begin{minipage}[c]{0.49\textwidth}
				\includegraphics[width=\textwidth]{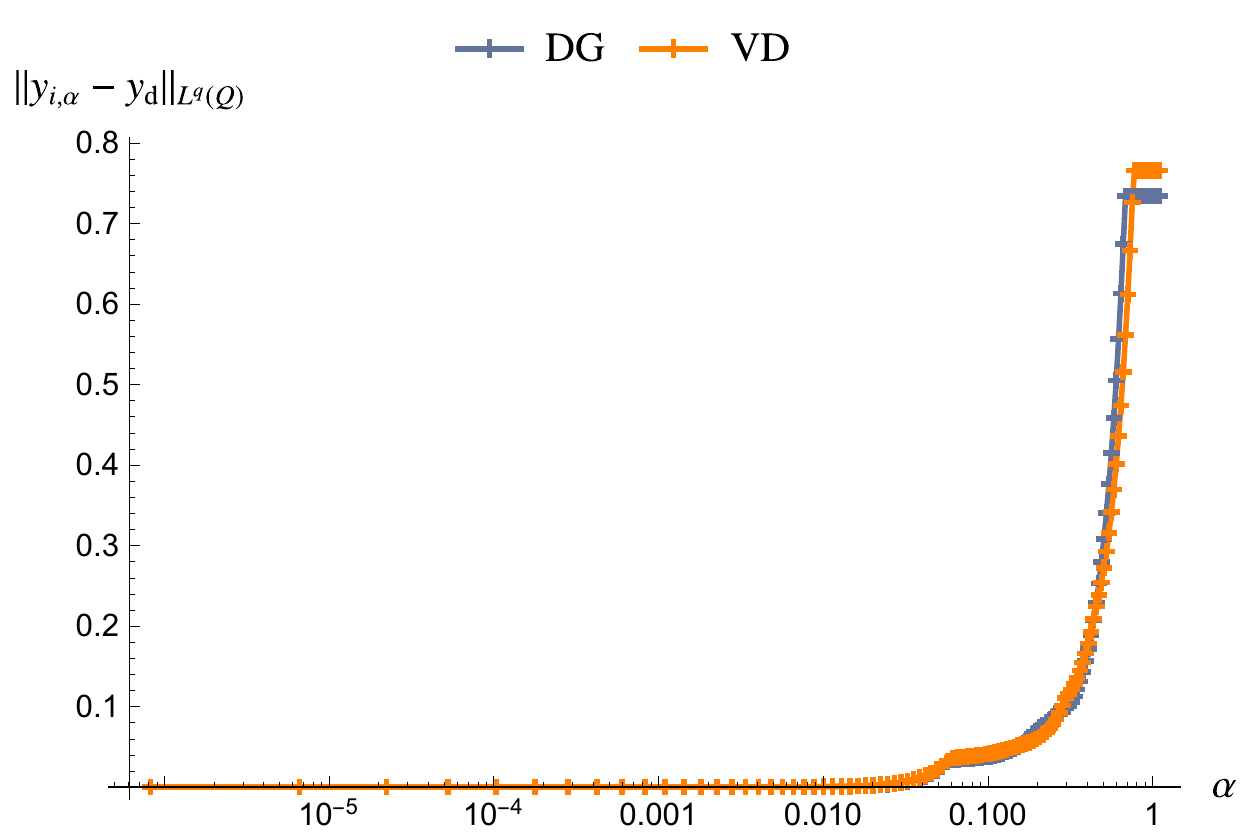}
			\end{minipage}
		\end{center}
		\caption{The dependence on the penalty parameter $\alpha$ of the measure norm of $u_{\sigma,\alpha}$ and $u_{\operatorname{DG},\alpha}$ (left) and the errors $y_{\sigma,\alpha} - y_{\operatorname{d}}$ and $y_{\operatorname{DG},\alpha} - y_{\operatorname{d}}$ in the $L^{4/3}$ norm (right).}
		\label{fig:uNorm_and_yerror}
	\end{figure}

For further comparison of the two discretization strategies, we choose a value of $\alpha$ that leads to a norm of the controls, that is neither zero nor maximal. For $\alpha = 0.456$, the reconstructed controls and states are displayed in Figure \ref{fig:NumericalResults}.
	\textcolor{M1}{
Here we see that the measure norm values $\|u_{\sigma,\alpha}\|_{\cM(\bar{Q}_c)} = 0.6610$ and $\|u_{\operatorname{DG,\alpha}}\|_{\cM(\bar{Q}_c)} = 0.6692$ both differ from the true value $\|u\|_{\cM(\bar{Q}_c)} =1$. Furthermore we observe that $\supp(u_{\sigma,\alpha}) = \{(\tfrac{1}{1},\tfrac{1}{2})\}$ and $\supp(u_{\operatorname{DG},\alpha})= \{\tfrac{1}{2}\}\times (\tfrac{1}{2},\tfrac{5}{8}]$, where the first coincides with the support of $u = \updelta_{(1/2,1/2)}$. 
The control in the discontinuous Galerkin discrete setting can also be represented by a Dirac measure. In order to do so, it has to be multiplied by $\tau$ and the location of the measure in the time interval has to be chosen.}

	
%
%
	%
	%
	%
	\begin{figure}[h]
		\begin{center}
			\setlength{\tabcolsep}{0pt}
			\begin{tabular}{|c|c|c|}
				\hline
				\textbf{setup}
				& \textbf{\eqref{eq:Pvd} with $\alpha = 0.456$}
				& \textbf{\eqref{eq:Psigma} with $\alpha = 0.456$}
				\\
				\hline
				\raisebox{-0.5\height}{	
					\includegraphics[width=\imgwidth]{CoarseGrid3_VD_Nh-256_Nt-768__TrueControl.png}
				}
				&
				\raisebox{-0.5\height}{	
					\includegraphics[width=\imgwidth]{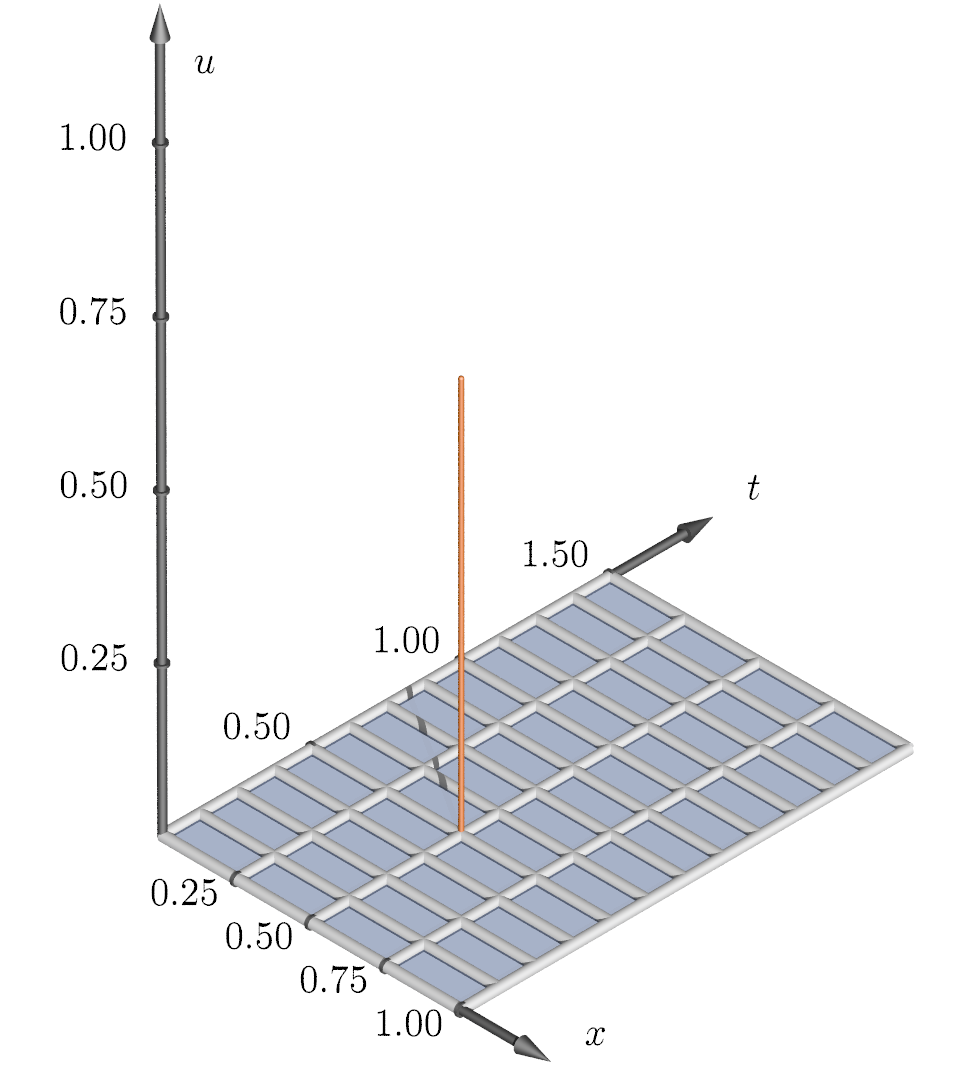}
				}
				&
				\raisebox{-0.5\height}{	
					\includegraphics[width=\imgwidth]{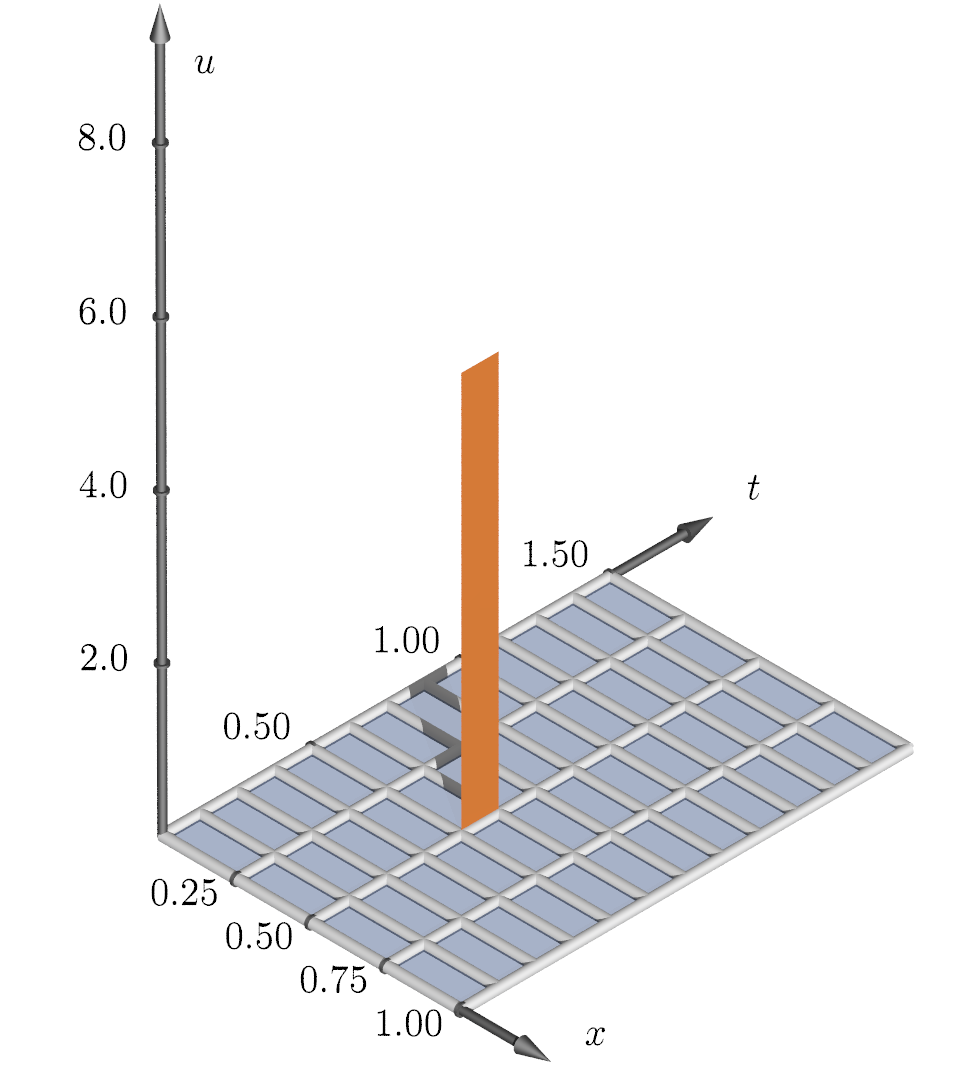}
				}
				\\
				\hline	
				\raisebox{-0.5\height}{
					\includegraphics[width=\imgwidth]{CoarseGrid3_VD_Nh-256_Nt-768__TrueState.png}
				}
				&
				\raisebox{-0.5\height}{	
					\includegraphics[width=\imgwidth]{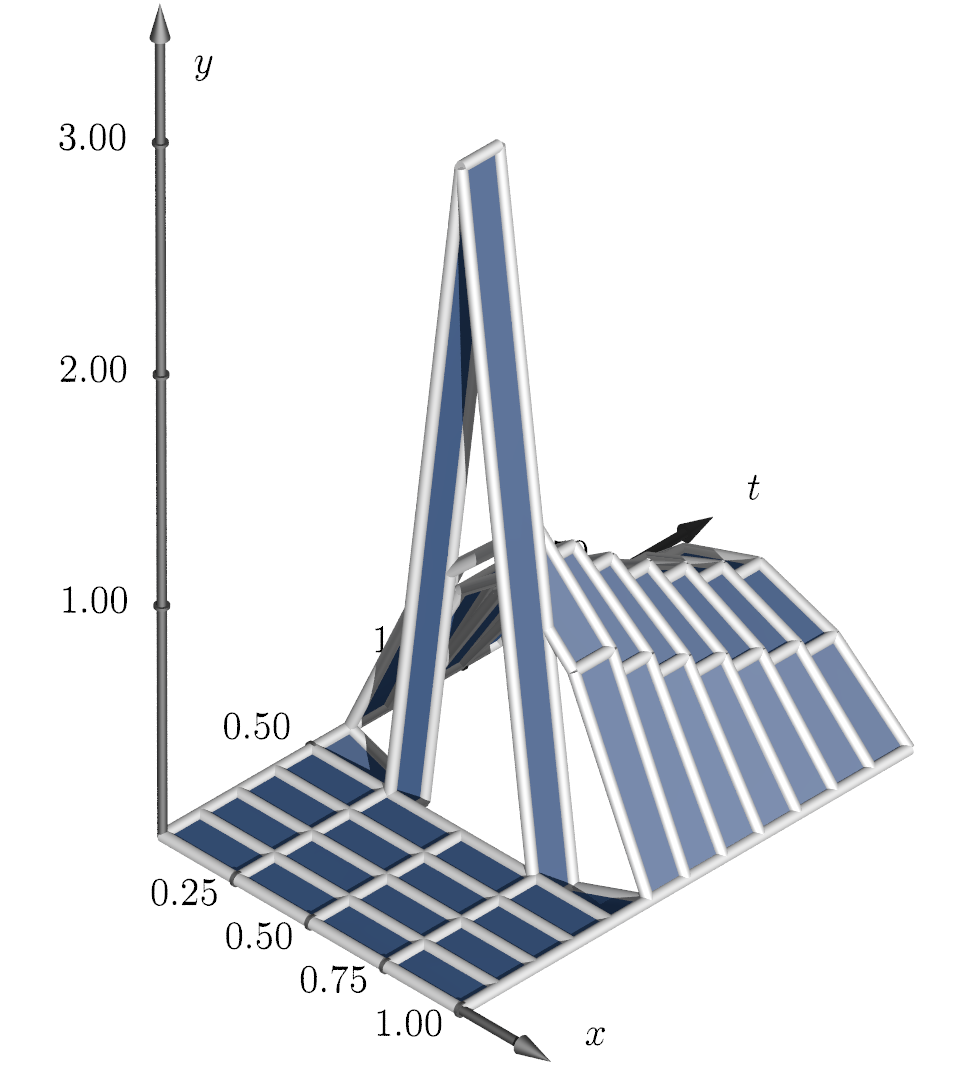}
				}
				&
				\raisebox{-0.5\height}{
					\includegraphics[width=\imgwidth]{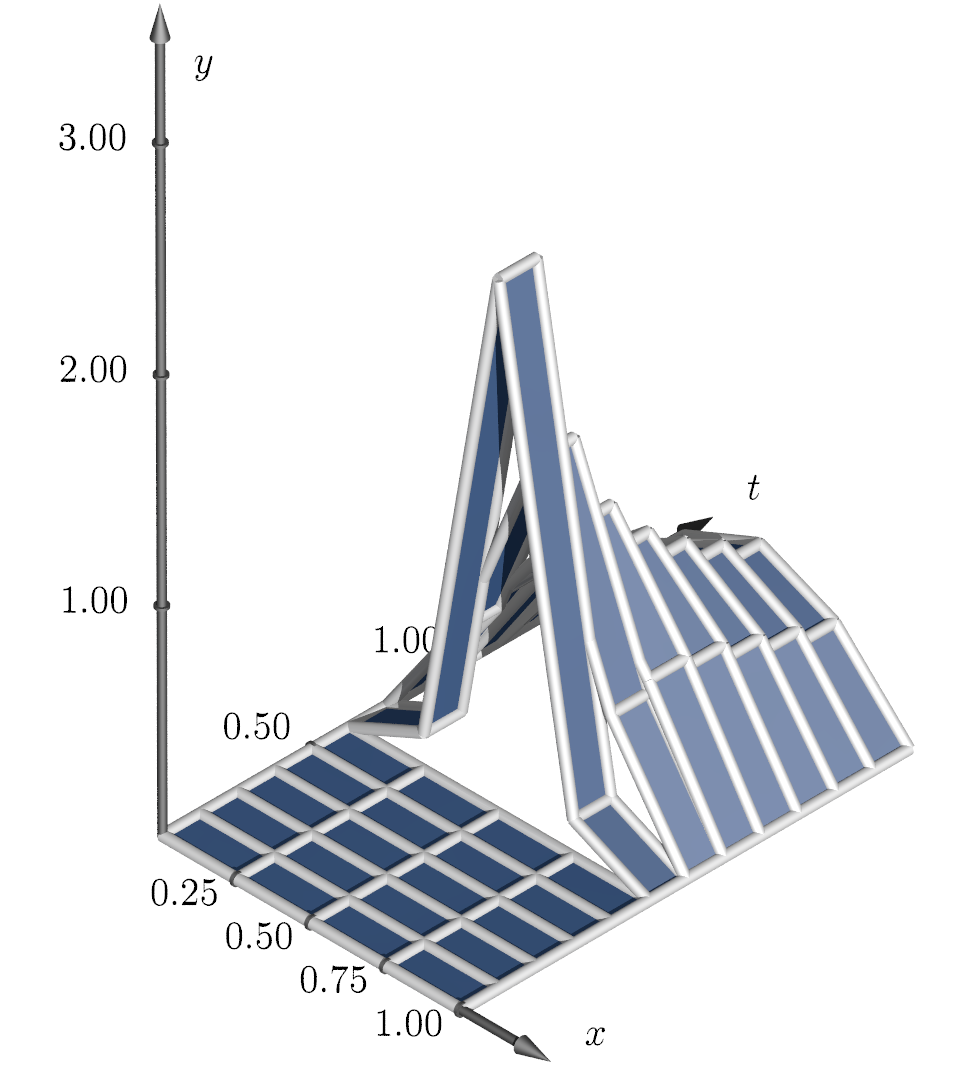}
				}	
				\\
				\hline
			\end{tabular}
		\end{center}
		\caption{
			\textbf{Top row:} The measure control and the optimal controls $u_{\sigma,0.456}$ and $u_{\operatorname{DG},0.456}$.
			\textbf{Bottom row:} The associated state $y(u)$ (sampled  from the analytic solution with spacial Fourier modes) and the associated states $y_{\sigma,0.456}$ and $y_{\operatorname{DG},0.456}$.
		}
		\label{fig:NumericalResults}
	\end{figure}

	If the control $u$ is not located on our space-time grid, it will be impossible to reproduce its support exactly. In the variational discretization approach a remedy might be choosing a test space $\cW_{\sigma}$ consisting of piecewise quadratic -- or even higher order -- functions in time. Thereby the maximal values of the test functions $\pm \alpha$ could be attained not only at grid points, but also inside the time intervals. Determining the location of these maximal values would mean to determine the exact position in time of the potential support of the control. This will be part of further research.
	
\bigskip
			
	\textcolor{C}{Our second numerical example aims at visualizing the convergence properties from \ref{thm:DiscVarConvergence} and \cite[Theorem 4.3.]{CasKun} for $|\sigma| \rightarrow 0$. Utilizing Fenchel duality, we can generate a discrete desired state $\DesiredState$ from a chosen true solution $u_{\operatorname{true}}$ for chosen $\bar{\alpha} >0$. From \eqref{eq:RecoverControl}, we deduce:
	\begin{equation*}
	\DesiredState = L^{- *} \Morrey^* u_{\operatorname{true}} - |L w_{\bar{\alpha}}|^{p-2} \, L w_{\bar{\alpha}}.
	\end{equation*}	}

	We set $u_{\operatorname{true}} = \updelta_{(1/2,1/2)}$ and the associated state $y_{\operatorname{true}} \coloneqq y(u_{\operatorname{true}})$ sampled from the analytic solution with spacial Fourier modes. 
	Furthermore, we take $w_{\bar{\alpha}}$, such that $w_{\bar{\alpha}}(1/2,1/2)= -\bar{\alpha}$ and for all other values $(x,t) \in \bar{Q}_c$ it holds $|w_{\bar{\alpha}}(x,t)| < \bar{\alpha}$. 
	For example, with $\bar{\alpha} = \tfrac{1}{4}$, this is fulfilled by
	\begin{equation*}
	w_{\bar{\alpha}}(x,t) \coloneqq -\tfrac{1}{4} \Bigparen{  \bigparen{ (t -0.5)^2 -1 }^2  \, \bigparen{ (2 x -1)^2 -1  }^2 }.
	\end{equation*}
	In the following we fix $\alpha = \bar{\alpha}$. The setup is visualized exemplary on an equidistant $4 \times 48$-grid in Figure \ref{fig:Generateyd}.

	\begin{figure}[t]
		\begin{center}
			\setlength{\tabcolsep}{0pt}
			\begin{tabular}{|c|c|c|c|}
				\hline
				\textbf{true control}
				& \textbf{associated state}
				& \textbf{adjoint state}
				& \textbf{desired state}
				\\
				\hline
				\raisebox{-0.5\height}{
					\includegraphics[width=0.235\textwidth]{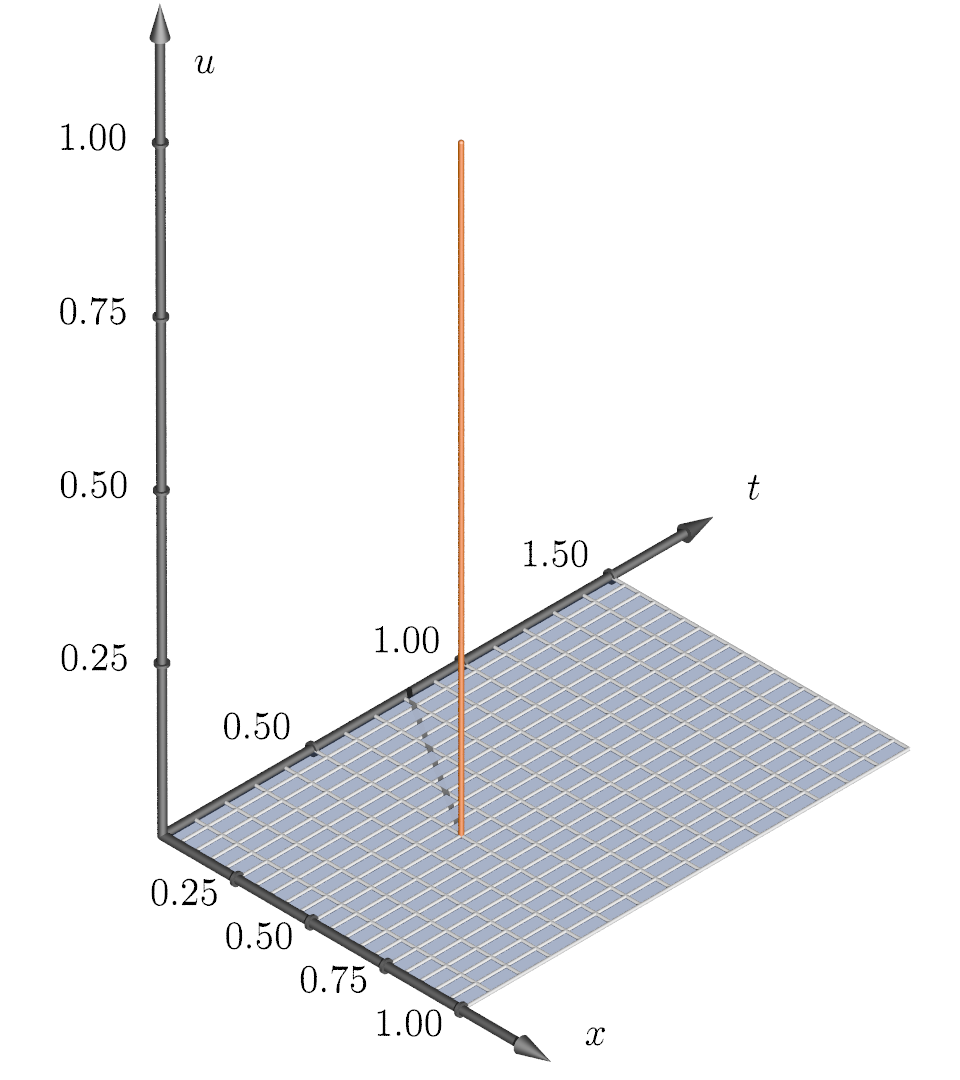}
				}
				&\raisebox{-0.5\height}{	
					\includegraphics[width=0.235\textwidth]{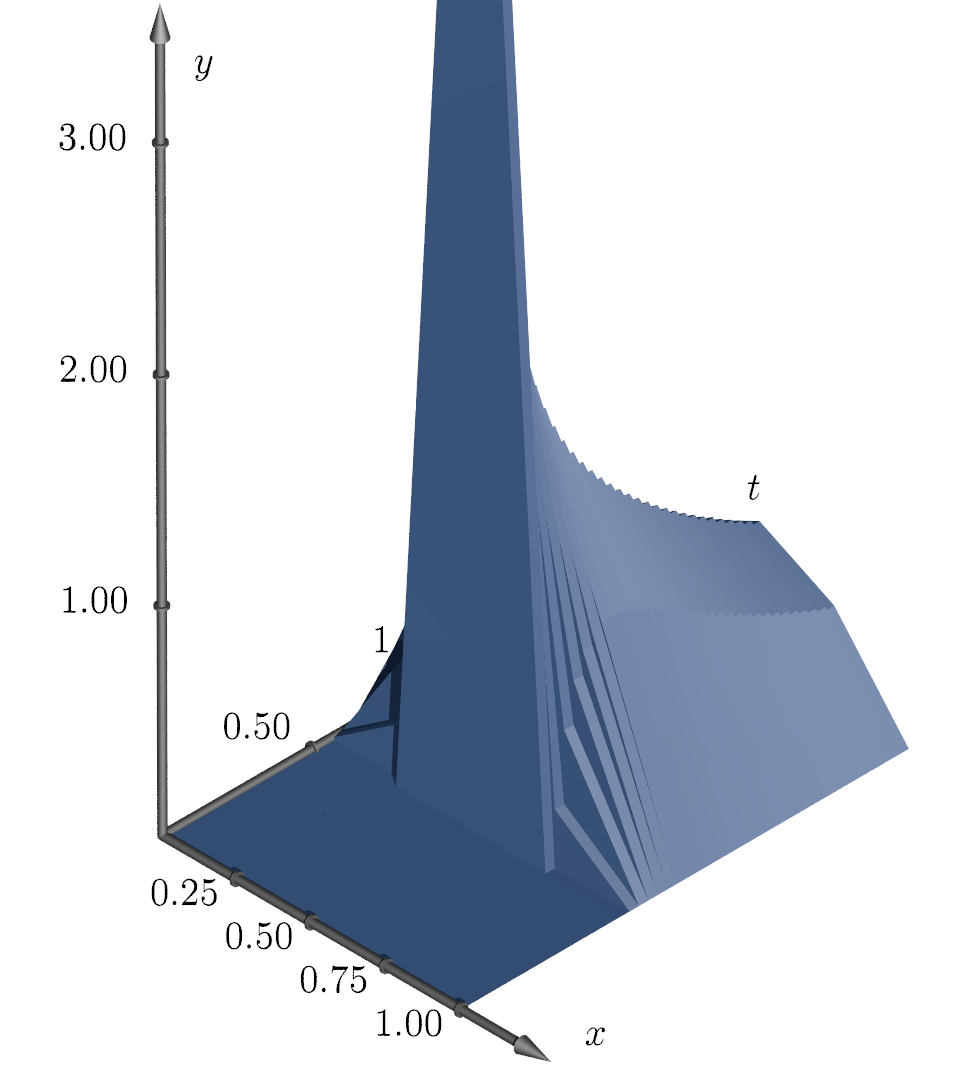}
				}
				&\raisebox{-0.5\height}{	
					\includegraphics[width=0.235\textwidth]{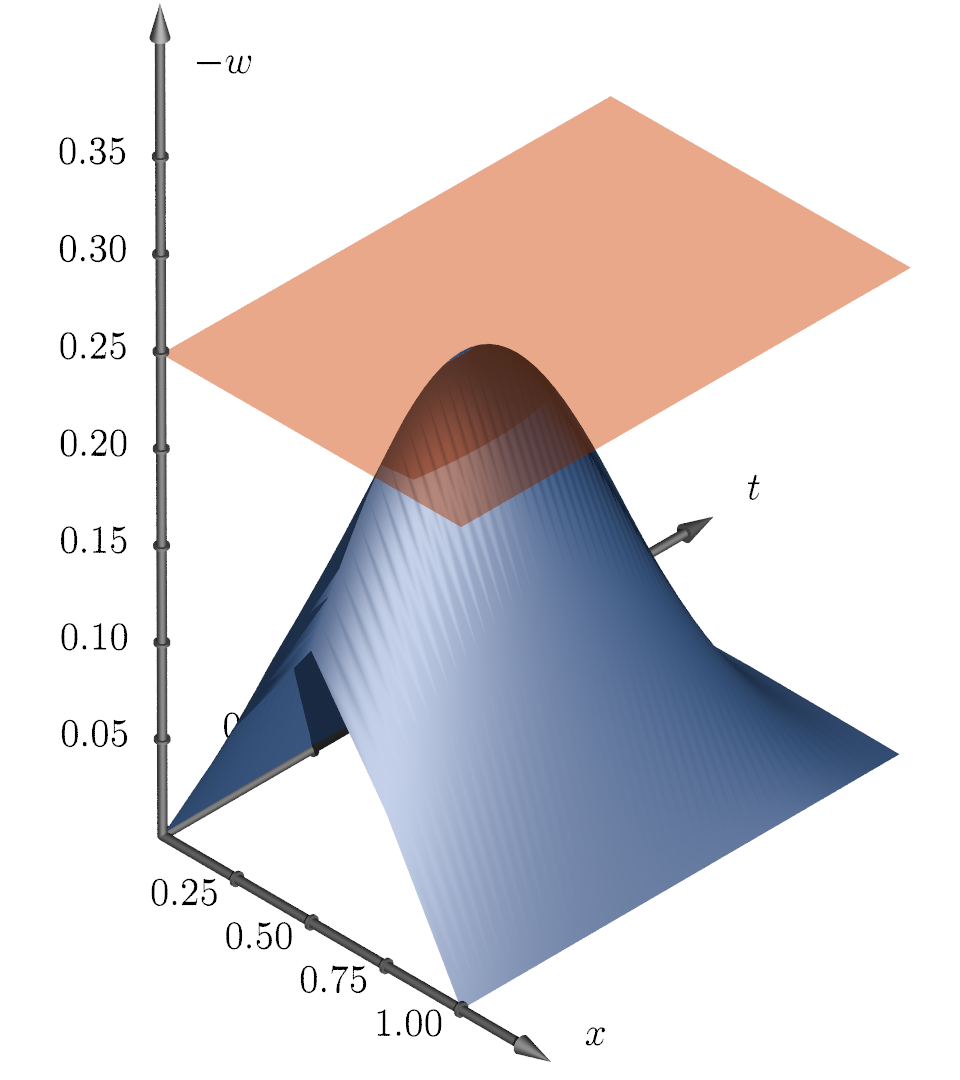}
				}
				&\raisebox{-0.5\height}{	
					\includegraphics[width=0.235\textwidth]{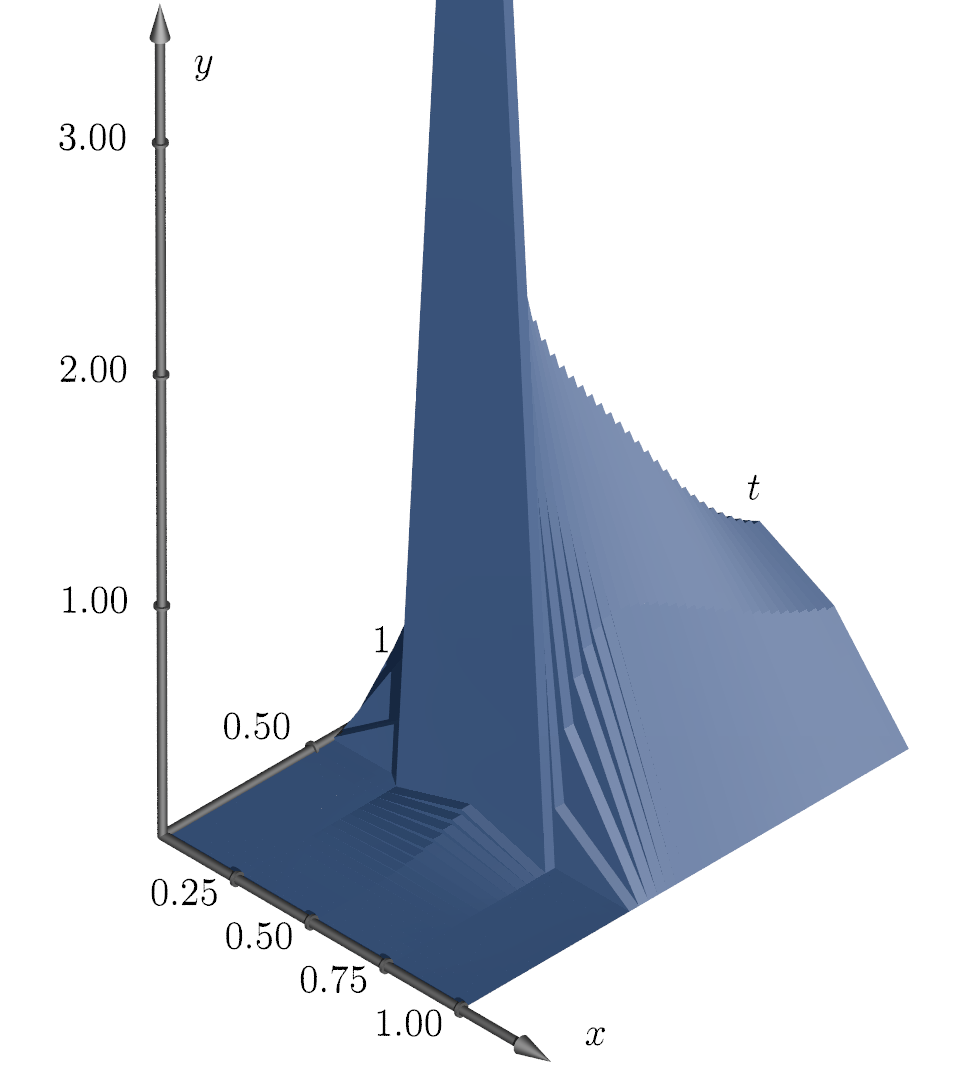}
				}
				\\
				\hline	
			\end{tabular}
		\end{center}
		\caption{Numerical setup on a $4 \times 48$ space-time grid with $q = \tfrac{4}{3}$. From left to right: true control $u_{\operatorname{true}} = \updelta_{(1/2,1/2)}$, interpolation of the associated state $y(u_{\operatorname{true}})$, adjoint state $w_{\bar{\alpha}}$ multiplied by $-1$ for easier visualization and desired state $\DesiredState$ calculated using Fenchel duality. 
		}
		\label{fig:Generateyd}
	\end{figure}
	
	For both discretization strategies, i.e. $i \in \{\sigma,\operatorname{DG}\}$, we have the following convergence properties:
	\begin{equation*}
	\lim\limits_{|\sigma|\rightarrow 0} \|u_{i}\|_{\cM(\bar{Q}_c)} = \|u_{\operatorname{true}}\|_{\cM(\bar{Q}_c)} \qquad \textrm{and} \qquad \lim\limits_{|\sigma|\rightarrow 0} \|y_{i} - y_{\operatorname{true}}\|_{L^q(Q)} = 0.
	\end{equation*}
	In \ref{fig:convergence} we $\log$-$\log$-plot the errors $|\, \|u_i\|_{\cM(\bar{Q}_c)} - \|u_{\operatorname{true}}\|_{\cM(\bar{Q}_c)}  |$ and $\| y_i - y_{\operatorname{true}} \|_{L^q(Q)}$, $i \in \{\sigma, \operatorname{DG}\}$
	versus the gridsize $h$.

	
		\begin{figure}[t]
		\begin{center}
			\setlength{\tabcolsep}{0pt}
			\begin{tabular}{|c|c|}
				\hline
				\textbf{$\tau = \tfrac{h}{2}$}
				& \textbf{$\tau = \tfrac{h^2}{2}$}
				\\
				\hline
				\raisebox{-0.1\height}{
					\includegraphics[width=0.4\textwidth]{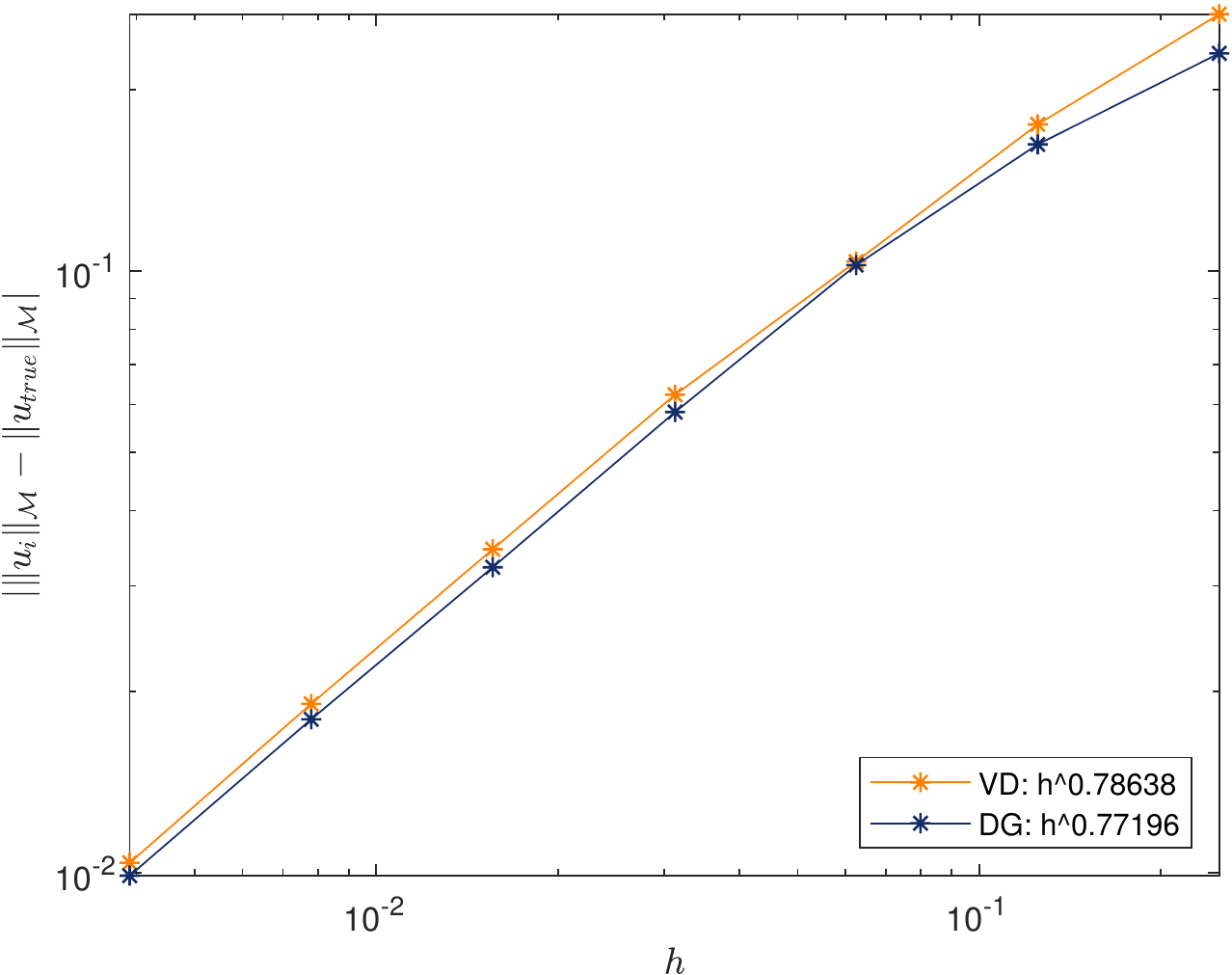}
				}
				&\raisebox{-0.1\height}{	
					\includegraphics[width=0.4\textwidth]{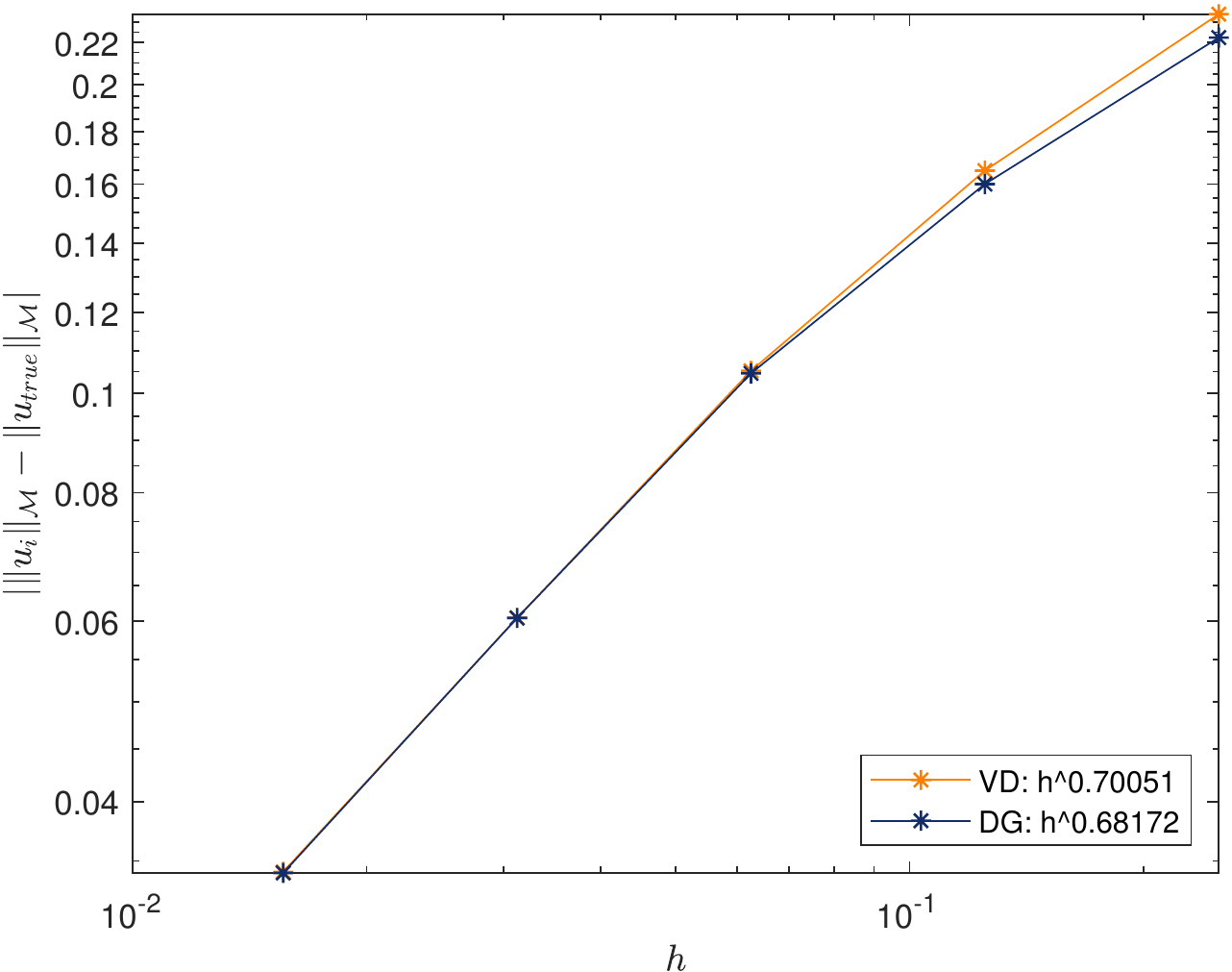}
				}
				\\[0.3cm]
				\hline
				\raisebox{-0.1\height}{ 
					\includegraphics[width=0.4\textwidth]{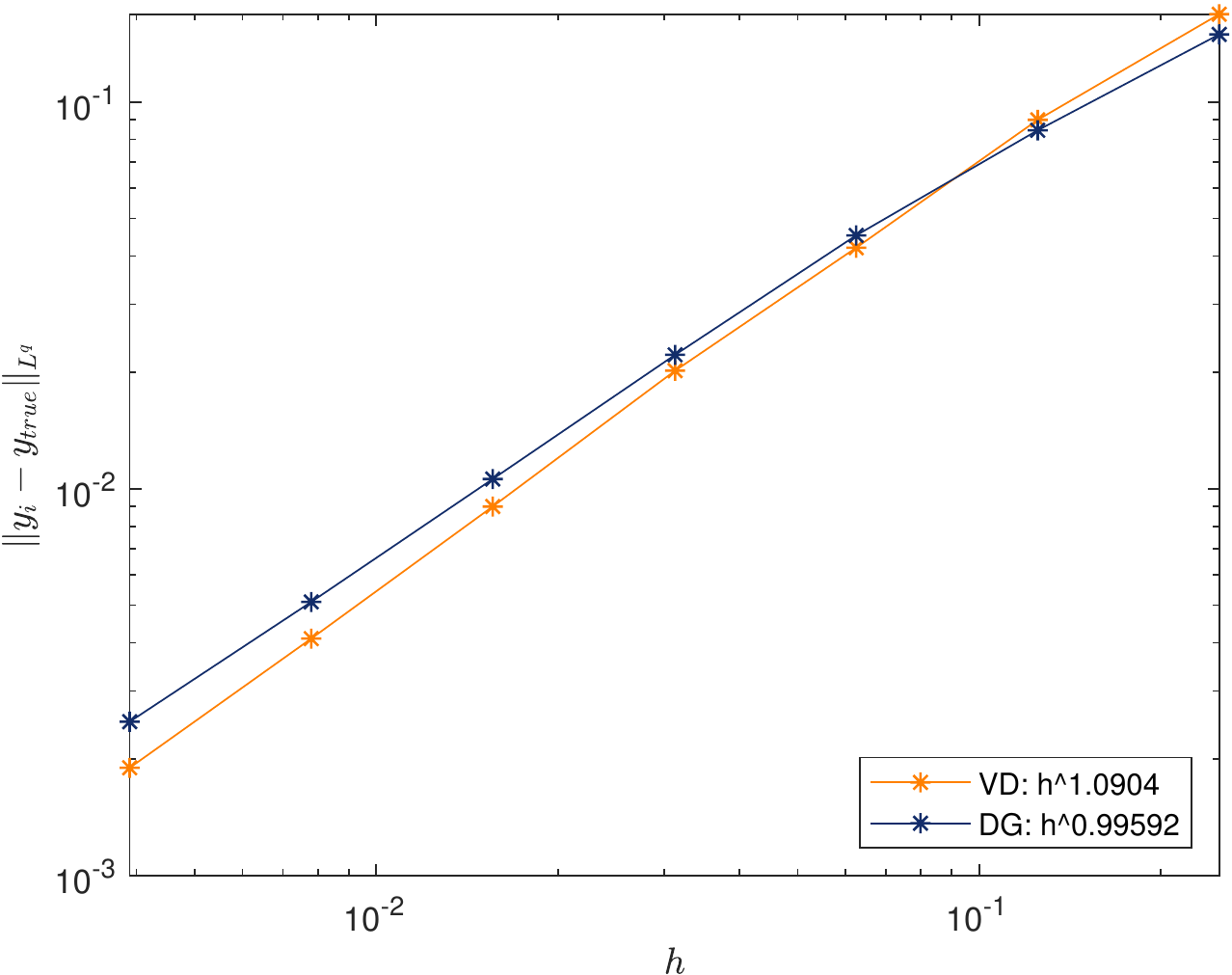} 
				}
				&\raisebox{-0.1\height}{
					\includegraphics[width=0.4\textwidth]{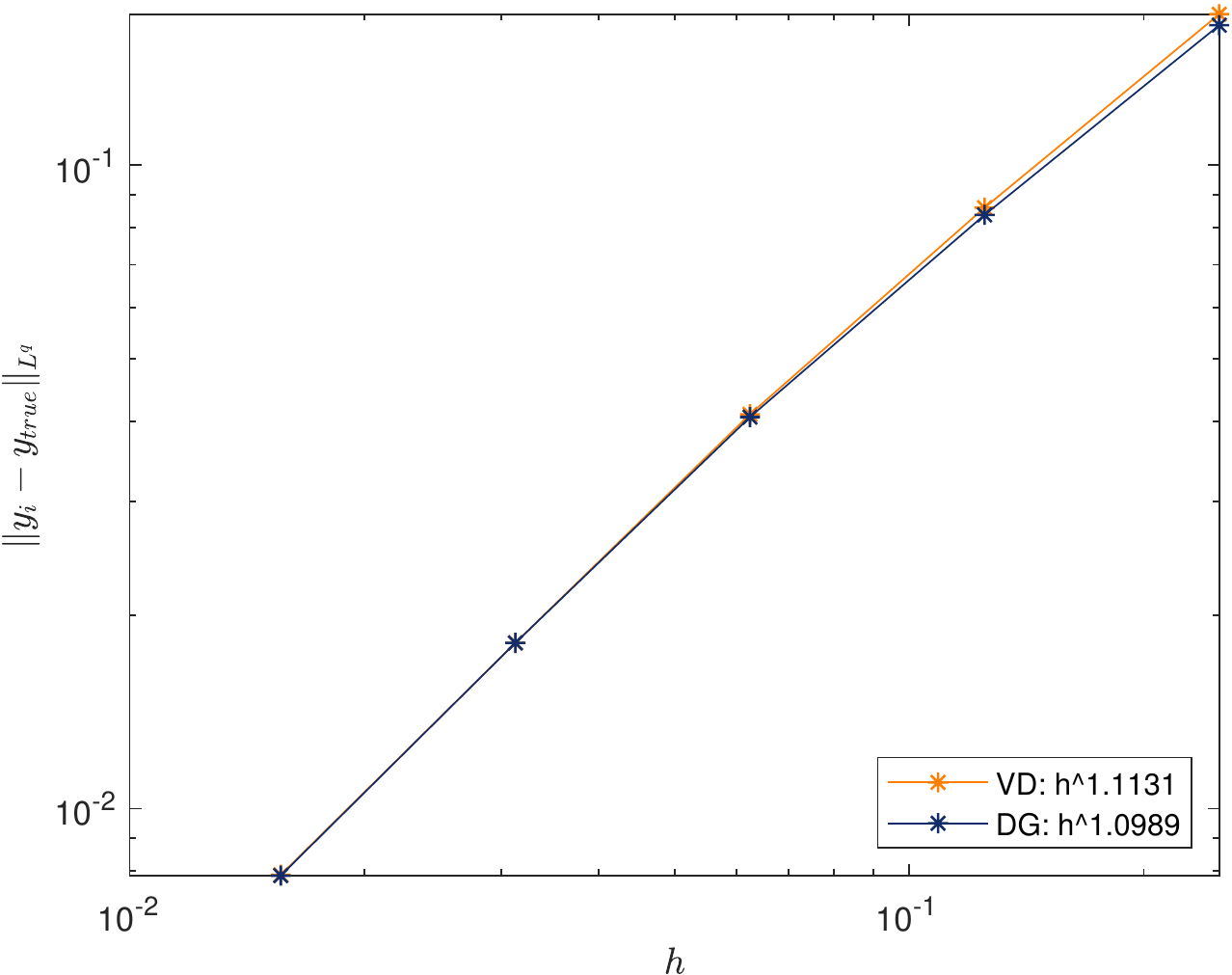} 
				}
				\\[0.3cm]
				\hline	
			\end{tabular}
		\end{center}
		\caption{\textbf{Top row:} The difference of the measure norms of the true control $u_{\operatorname{true}}$ and calculated optimal control $u_i$, $i \in \{\sigma, \operatorname{DG}\}$ for $\tau = \tfrac{h}{2}$ (left) and $\tau= \tfrac{h^2}{2}$ (right). \textbf{Bottom row:} The $L^{4/3}$ norm of the difference of the associated states $y_i - y_{\operatorname{true}}$, $i \in \{\sigma, \operatorname{DG}\}$ for $\tau = \tfrac{h}{2}$ (left) and $\tau= \tfrac{h^2}{2}$ (right).
		}
		\label{fig:convergence}
	\end{figure}
	
	The proven convergence properties can be observed, as the errors go to zero for $h \rightarrow 0$, which is equivalent to $|\sigma|\rightarrow 0$ since $\tau$ is always linked to $h$. \textcolor{M1}{It is interesting to mention that in the variational discrete setting, oscillations in the state $y_{\sigma}$ occur for small gridsizes, where $\tau = \tfrac{h}{2}$. This is caused by the Crank-Nicolson-like scheme. Nevertheless, we see convergence also in this case.}
	In the legend of the plots we display the slope in orders of $h$ of linear fitting curves of the errors, respectively. Proving general convergence rates will be part of further research.
	
	\bigskip
	We observe many similarities in the derivation of the algorithms to solve the discrete problems. 
	The implementation and the level of difficulty in programming is comparable for both approaches and we also observe similar iteration counts. 
	The main advantage of the variational discretization compared to the discontinuous Galerkin discretization is the maximal discrete sparsity of the control achieved by choosing a suitable Petrov-Galerkin-ansatz and -test space.
	
	\bigskip 
	\textbf{Acknowledgment.} We thank the referees for their helpful remarks that improved the manuscript.

%


\printbibliography

\end{document}